\documentclass[a4paper,reqno]{amsart}

\usepackage{musicography}
\usepackage{comment}
\usepackage[T1]{fontenc}
\usepackage[utf8x]{inputenc}
\usepackage[main=english,german]{babel}
\usepackage{yfonts}
\usepackage{dsfont}
\usepackage{amscd,amssymb,amsmath,amsthm,amsfonts}
\usepackage{mathtools}
\usepackage{accents}
\usepackage{float}
\usepackage{tensor}
\usepackage{caption}
\usepackage{graphicx}
\usepackage{tikz}
\usetikzlibrary{calc,matrix,arrows,decorations.pathmorphing}
\usepackage{calc}
\usepackage[cal=boondox,scr=boondoxo]{mathalfa}
\usepackage{enumitem}
\usepackage{marginnote}
\usepackage{booktabs}
\usepackage{scalerel}[2016/12/29]
\usepackage{url}
\usepackage{contour}
\usepackage[normalem]{ulem}

\usepackage{hyperref}
\hypersetup{colorlinks=true,pageanchor=false,linkcolor=blue,citecolor=red,urlcolor=red}
\usepackage[all]{hypcap}

\allowdisplaybreaks

\newtheorem{theorem}{Theorem}[section]
\newtheorem{corollary}[theorem]{Corollary}
\newtheorem{proposition}[theorem]{Proposition}
\newtheorem{lemma}[theorem]{Lemma}

\theoremstyle{definition}

\theoremstyle{remark}
\newtheorem{remark}[theorem]{Remark}

\newtheorem{conjecture}[theorem]{Conjecture}

\newcommand{\Z}{\mathbb{Z}}

\newcommand{\R}{\mathbb{R}}
\newcommand{\C}{\mathbb{C}}

\newcommand{\bbC}{\mathbb{C}}

\newcommand{\bfM}{\boldsymbol{M}}

\newcommand{\bfSigma}{\boldsymbol{\Sigma}}

\newcommand{\calH}{\mathcal{H}}
\newcommand{\calI}{\mathcal{I}}

\newcommand{\calL}{\mathcal{L}}

\newcommand{\fraki}{\mathfrak{i}}

\newcommand{\fraku}{\mathfrak{u}}

\newcommand{\frakU}{\mathfrak{U}}

\renewcommand{\epsilon}{\varepsilon}
\renewcommand{\theta}{\vartheta}
\renewcommand{\phi}{\varphi}
\renewcommand{\Gamma}{\varGamma}
\renewcommand{\Sigma}{\varSigma}

\newcommand{\id}{\mathrm{id}}

\DeclareMathOperator{\Forall}{\forall}

\DeclareMathOperator{\bcs}{\natural}

\newcommand{\leqs}{\leqslant}
\newcommand{\geqs}{\geqslant}

\newcommand{\mods}[1]{\operatorname{\mathnormal{#1}-mod}}

\newcommand{\fsl}{\mathfrak{sl}}
\newcommand{\Mod}{\mathrm{Mod}}

\newcommand{\BM}{\mathrm{BM}}

\newcommand{\PGL}{\mathrm{PGL}}

\newcommand{\Hom}{\mathrm{Hom}}

\newcommand{\End}{\mathrm{End}}

\newcommand{\sig}{n}
\newcommand{\Cob}{\mathrm{Cob}}
\newcommand{\Tan}{\mathrm{Tan}}
\newcommand{\lk}{\mathrm{lk}}

\newcommand{\sqbinom}[2]{\left[ \begin{matrix} #1 \\ #2 \end{matrix} \right]}

\makeatletter
\newcommand{\subalign}[1]{
  \vcenter{
    \Let@ \restore@math@cr \default@tag
    \baselineskip\fontdimen10 \scriptfont\tw@
    \advance\baselineskip\fontdimen12 \scriptfont\tw@
    \lineskip\thr@@\fontdimen8 \scriptfont\thr@@
    \lineskiplimit\lineskip
    \ialign{\hfil$\m@th\scriptstyle##$&$\m@th\scriptstyle{}##$\crcr
      #1\crcr
    }
  }
}
\makeatother

\def\clap#1{\hbox to 0pt{\hss#1\hss}}

\def\mathrlap{\mathpalette\mathrlapinternal}

\def\mathrlapinternal#1#2{%
\rlap{$\mathsurround=0pt#1{#2}$}}

\newcommand{\homol}{\calH}

\newcommand{\pic}[2][0]{\raisebox{-0.5\height + 2.5pt + #1pt}{\includegraphics{Pictures/#2.pdf}}}

\newcommand\arxiv[2]{\href{https://arXiv.org/abs/#1}{\texttt{arXiv:\allowbreak #1} #2}}
\newcommand\doi[2]{\href{https://doi.org/#1}{#2}}

\contourlength{0.625pt}

\DeclareRobustCommand{\myuline}[1]{
 \ifmmode \text{\uline{$\phantom{#1}$}\llap{\contour{white}{$#1$}}}
 \else \uline{\phantom{#1}}\llap{\contour{white}{#1}} \fi
}

\newcommand{\transm}[1]{{\myuline{#1} \hspace*{1pt}}}

\makeatletter
\def\namedlabel#1#2{\begingroup
    #2%
    \def\@currentlabel{#2}%
    \phantomsection\label{#1}\endgroup
}
\makeatother

\newcommand{\formv}{\mathbb{X}}
\newcommand{\cohom}{\vartheta}

\newcommand{\dmnsnl}[1]{$#1$-di\-men\-sion\-al}

\begin{document}

\raggedbottom

\title{Hennings TQFTs for Cobordisms Decorated With Cohomology Classes}

\author[M. De Renzi]{Marco De Renzi} 
\address{IMAG, Université de Montpellier, Place Eugène Bataillon, 34090 Montpellier, France}
\email{marco.de-renzi@umontpellier.fr}

\author[J. Martel]{Jules Martel} 
\address{AGM, Cergy Paris Université, Site de Saint-Martin, 2 avenue Adolphe Chauvin, 95300 Pontoise, France} 
\email{jules.martel-tordjman@cyu.fr}

\author[B. Wang]{Bangxin Wang}
\address{Institute of Mathematics, University of Zurich, Winterthurerstrasse 190, CH-8057 Zurich, Switzerland}
\email{bangxin.wang@math.uzh.ch}

\begin{abstract}
 Starting from an abelian group $G$ and a factorizable ribbon Hopf $G$-bialgebra $H$, we construct a TQFT $J_H$ for connected framed cobordisms between connected surfaces with connected boundary decorated with cohomology classes with coefficients in $G$. When restricted to the subcategory of cobordisms with trivial decorations, our functor recovers a special case of Kerler--Lyubashenko TQFTs, namely those associated with factorizable ribbon Hopf algebras. Our result is inspired by the work of Blanchet--Costantino--Geer--Patureau, who constructed non-semisimple TQFTs for admissible decorated cobordisms using the unrolled quantum group of $\fsl_2$, and by that of Geer--Ha--Patureau, who reformulated the underlying invariants of admissible decorated $3$-manifolds using ribbon Hopf $G$-coalgebras. Our work represents the first step towards a homological model for non-semisimple TQFTs decorated with cohomology classes that appears in a conjecture by the first two authors.
\end{abstract}

\maketitle
\setcounter{tocdepth}{3}

\section{Introduction}\label{S:intro}

The interplay between the algebraic structure of monoidal categories and the topological properties of low-di\-men\-sion\-al geometric objects provides a cornerstone for the field of \textit{quantum topology}. One way to frame this interaction is that quantum topology is a discipline that develops an algebraic approach, based on generators and relations, to low-dimensional topology. In particular, one of its purposes is to produce topological invariants of manifolds, and of maps between manifolds, out of algebras, their representations, and the monoidal categories they form. A more ambitious goal is typically to organize all these invariants into coherent structures known as a \textit{TQFTs} (\textit{Topological Quantum Field Theories}).

In this paper, we provide a TQFT construction that produces functors defined over categories of cobordisms decorated with cohomology classes. In order to do this, we fix an abelian group $G$, and we introduce the notion of a \textit{ribbon Hopf $G$-bialgebra} (see Section~\ref{S:ribbon_structure}), by generalizing Ohtsuki's colored ribbon Hopf algebras \cite{O93} and (the abelian case of) Virelizier's ribbon Hopf $G$-coalgebras \cite{V00}. Roughly speaking, such a ribbon Hopf $G$-bialgebra $H$ is given by a family of vector spaces $H_\alpha^\beta$ parametrized by two indices $\alpha, \beta \in G$. The direct sum 
\[
 H_0^\oplus = \bigoplus_{\alpha \in G} H_0^\alpha
\]
is naturally a ribbon Hopf algebra, and so is $H_0^0$. The main example of ribbon Hopf $G$-bialgebra we will consider in this paper arises from the quantum group of $\fsl_2$ with $G = \C/2\Z$, and is introduced in Section~\ref{S:quantum_sl2_algebra} by adapting the example discussed in \cite{GHP22}. We also consider the category $3\Cob^G$ of \textit{$G$-decorated connected framed cobordisms} (see Section~\ref{S:cobordisms}), whose morphisms come equipped with cohomology classes with $G$-coefficients. Here is our main result, see also Theorem~\ref{T:main}.

\begin{theorem}\label{T:introduction}
 Every factorizable\footnote{See Section~\ref{S:factorizability} for a definition.} ribbon Hopf $G$-bialgebra $H$ induces a braided monoidal functor
 \[
  J_H : 3\Cob^G \to \mods{H_{\mathrm{0}}^\oplus}.
 \]
\end{theorem}

As an immediate consequence, when restricted to the subcategory $3\Cob$ of cobordisms decorated with the trivial cohomology class, our functor recovers Kerler and Lyubashenko's TQFT associated with the factorizable ribbon Hopf algebra $H_0^0$, compare with Corollary~\ref{C:relation_with_KL}.

\begin{corollary}\label{C:introduction}
 If $H$ is a factorizable ribbon Hopf $G$-bialgebra, then the TQFT $J_H$ of Theorem~\ref{T:introduction} fits into the commutative diagram of braided monoidal functors
 \begin{center}
  \begin{tikzpicture}[descr/.style={fill=white}] \everymath{\displaystyle}
   \node (P0) at (0,0) {$3\Cob$};
   \node (P1) at (2,0) {$\mods{H_{\mathrm{0}}^{\mathrm{0}}}$};
   \node (P2) at (0,-1) {$3\Cob^G$};
   \node (P3) at (2,-1) {$\mods{H_{\mathrm{0}}^\oplus}$};
   \draw
   (P0) edge[->] node[above] {\scriptsize $J_{H_0^0}$}(P1)
   (P0) edge[->] node[left] {\scriptsize $\iota$} (P2)
   (P3) edge[->] node[right] {\scriptsize $\rho$} (P1)
   (P2) edge[->] node[below] {\scriptsize $J_H$} (P3);
  \end{tikzpicture}
 \end{center}
 where $J_{H_0^0}$ is the Kerler--Lyubashenko TQFT, $\iota$ is the inclusion functor, and $\rho$ is the restriction functor.
\end{corollary}

\subsection{State of the art}

As their name suggests, TQFTs have a physical interpretation, and they were first axiomatized in mathematical terms by Atiyah \cite{A88}, who defined them, broadly speaking, as monoidal functors from categories of cobordisms to categories of modules over a given ring or field. Morphisms in categories of cobordisms are manifolds (possibly equipped with embedded links or other decorations) whose boundary is decomposed into an incoming part, or source, and an outgoing one, or target. The monoidal structure for cobordisms is traditionally induced by disjoint union, according to Atiyah's definition, although other meaningful frameworks exist (like the one considered in this paper, for instance). The first full-fledged construction of TQFTs was performed by Turaev \cite{T94} using \textit{semisimple modular categories}. Turaev's work generalized and extended previous results by Reshetikhin and Turaev \cite{RT91}, who first defined invariants of $3$-manifolds using semisimple quotients of categories of representations of ribbon Hopf algebras such as \textit{quantum groups}. Their construction was inspired by previous work of Witten on the Jones polynomial and Chern--Simons quantum field theory \cite{W89}, so the resulting invariants are now known as \textit{WRT invariants} (for \textit{Witten--Reshetikhin--Turaev}). Since the semisimplicity assumption played a crucial role in Turaev's construction, the associated functors are sometimes referred to as \textit{semisimple} TQFTs. The first construction of topological invariants of closed $3$-manifolds built from possibly non-semisimple ribbon Hopf algebras is due to Hennings \cite{H96}. In his approach, Hennings requires some assumptions on ribbon Hopf algebras, which need to be \textit{unimodular} and \textit{twist non-degenerate}. The resulting invariants were deeply generalized by Lyubashenko \cite{L94} to families of representations of mapping class groups of surfaces using potentially non-semisimple \textit{modular categories}, such as categories of representations of \textit{factorizable} ribbon Hopf algebras. Later, Kerler and Lyubashenko extended these mapping class group representations to \textit{EQTFTs} (\textit{Extended TQFTs}) for connected cobordisms \cite{KL01}. In particular, such ETQFTs contain TQFT functors defined on a category of cobordisms whose objects are connected surfaces with connected boundary, whose morphisms are connected cobordisms between them, and whose monoidal structure is induced by boundary connected sum. Such functors are called \textit{Kerler--Lyubashenko TQFTs} in this paper.

More recently, strong invariants of closed $3$-manifolds decorated with cohomology classes, known as \textit{CGP invariants}, were constructed by Costantino, Geer, and Patureau in \cite{CGP12}. This was done using the theory of \textit{modified traces} and \textit{relative modular categories}, which are a non-semisimple graded generalization of modular categories. These invariants were extended to TQFTs, known as \textit{BCGP TQFTs}, by Blanchet, Costantino, Geer, and Patureau \cite{BCGP14} for the unrolled quantum group of $\fsl_2$, and to ETQFTs by the first author \cite{D17} for general relative modular categories. Both approaches follow the \textit{universal construction} of Blanchet, Habegger, Masbaum, and Vogel \cite{BHMV95} to produce functors defined on categories of so-called \textit{admissible} cobordisms decorated with cohomology classes. By contrast with Kerler and Lyubashenko's framework, such categories feature disconnected objects, and disjoint union as tensor product. As we shall discuss later, cohomology classes should not be seen as a technical nuisance of the theory, but rather witness the richness of such invariants. Modified traces were later integrated in Hennings and Lyubashenko's constructions to obtain invariants, known as \textit{renormalized Hennings} and \textit{renormalized Lyubashenko} invariants respectively, that extend to disconnected admissible cobordisms on the level of both TQFTs \cite{DGGPR19} and ETQFTs \cite{D21}. More recently, constructions exploiting non-semisimple spherical categories were also developed, see \cite{CGPV23}.

\subsection{Main results}

As it is clear from the previous discussion, several approaches to the construction of non-semisimple TQFTs are now available. In the present work, we will essentially follow Kerler and Lyubashenko, but we will generalize their construction to obtain a monoidal functor defined on the category $3\Cob^G$ of connected framed cobordisms decorated with cohomology classes with $G$-coefficients. As our main tool, we will use a factorizable ribbon Hopf $G$-bialgebra $H$, see Theorem~\ref{T:main}. For cobordisms decorated with the trivial cohomology class, our functor restricts to the one associated by Kerler and Lyubashenko's construction to the factorizable ribbon Hopf algebra $H_0^0$, see Corollary~\ref{C:relation_with_KL}. 

To motivate our interest in this generalized construction, we highlight the fact that,  for the CGP invariants associated with the unrolled quantum group of $\fsl_2$, the dependence on a cohomology class with $\C/2\Z$-coefficients was crucially used to recover the classification of lens spaces \cite[Proposition~6.24]{BCGP14}. It is known that this cannot be done using WRT invariants, and it is currently not known whether this can be done using renormalized Lyubashenko invariants associated with the small quantum group of $\fsl_2$.

Furthermore, a TQFT functor automatically yields representations of mapping class groups of the objects of the corresponding cobordism category. It is well-known that representations  coming from semisimple TQFTs associated with quantum $\fsl_2$ are never faithful, since Dehn twists are always sent to matrices of finite order. This obstruction vanishes for representations arising from all non-simimple TQFTs associated with quantum $\fsl_2$. This observation makes them candidates for tackling the question of linearity of mapping class groups, since no non-trivial elements have currently been found in the kernels of these representations. Working with manifolds endowed with cohomology classes requires either looking at \textit{twisted representations} or restricting to diffeomorphisms that fix a given cohomology class. In first approximation, it is completely acceptable to restrict to representations of \textit{Torelli subgroups}, since their elements fix every cohomology class, and since such subrgoups are at the heart of the linearity problem (indeed, by their very definition, they constitute precisely the part of the mapping class group that is not detected by the standard homological representation). Since these groups are of finite type, finitely many different cohomology classes can in principle be used to detect finite families of generating diffeomorphisms. In such a situation, a continuity argument in $\C/2\Z$ could be used to deduce the existence of a single cohomology class detecting every element of the Torelli group. This observation motivates the interest in representations of Torelli groups that depend on the choice of a cohomology class with $\C/2\Z$-coefficients. We notice however that, in order to apply such a continuity argument, we should first prove that these representations depend continuously on the evaluation of the cohomology class against a finite basis of the homology of the surface.

In \cite{DM22}, the first two authors have proposed a completely different perspective on Lyubashenko's mapping class groups representations associated with small quantum $\fsl_2$, by developing a model based on twisted homology groups of configuration spaces. Such a model also provides a homological construction of the action of quantum $\fsl_2$. This framework is promising for studying faithfulness of mapping class group representations, since it is the same that was used by Bigelow \cite{B00} to prove linearity of braid groups. For punctured discs, the fact that homological representations are isomorphic to quantum representations (possibly decorated with cohomology classes) was proved in \cite{M20}. In \cite[Section~6.3.1]{DM22}, homological representations in higher genus are upgraded to twisted representations for surfaces endowed with cohomology classes, and \cite[Conjecture~6.6]{DM22} states that such representations should also arise from a non-semisimple TQFT construcion. The functor of Theorem~\ref{T:main} yields precisely the representations that, according to our expectation, should provide the quantum counterpart to such twisted homological representations. Now that these representations have actually been constructed, we rephrase the conjecture in Section~\ref{S:J_H_is_homological}, to be studied in a future paper. The homological setup pinpoints bases in which coefficients of these representations depend on the cohomology class as polynomials over the integers (see \cite[Section~6.3.1]{DM22}). This integrality property is a crucial strength of semisimple TQFTs, see for instance the discussion of \cite[Section~1.2]{DM22}. Its first appearance in a non-semisimple context is \cite[Corollary~6.3]{DM22}, and a polynomial dependence of BCGP representations in the coefficients of the cohomology class has not been established in general yet (although an analytic dependence has been observed in specific cases).

In the present paper, we use as our main tool for the construction a factorizable ribbon Hopf $G$-bialgebra $H$. In some sense, such an algebraic structure is the appropriate one for encoding cohomology classes with $G$-coefficients on cobordisms. The inspiration for this algebraic setup comes from the work of Virelizier \cite{V00} and Geer, Ha, and Patureau \cite{GHP22} on ribbon Hopf group-coalgebras. In order to construct a TQFT functor, the notion of a ribbon Hopf group-bialgebra is required, as explained in Section~\ref{S:Construction_of_the_functor}. In the case of unrolled quantum groups, an improved version of the invariants provided by our TQFTs was defined in \cite{GHP20} using modified integrals, and it was shown to recover the corresponding CGP invariants. Applying the universal construction to these invariants would thus immediately yield TQFT functors, as it was shown in \cite{D17}. However, the invariants of \cite{GHP20} can be extended to a larger category of admissible decorated cobordisms than the one used for BGCP TQFTs, by allowing embedded \textit{bichrome graphs} as defined in \cite[Section~5.1]{GHP20}. The TQFTs resulting from the application of the universal construction to this larger category of admissible decorated cobordisms are expected to be larger than BCGP TQFTs, and our functors would correspond to the restriction of these larger TQFTs to the category of connected cobordisms, by analogy with the relation established in the absence of decorations in \cite[Appendix~C]{DM22}.

\subsection{Structure of the paper}

In Section~\ref{S:cobordisms}, we introduce the category $3\Cob^G$ of connected framed cobordisms decorated with cohomology classes with coefficients in an abelian group $G$. This will provide the source of our TQFT functor. In Subsection~\ref{S:top_tangles}, we propose a diagrammatic model for these decorated cobordisms based on the category $\Tan^G$ of $G$-labeled top tangles in $G$-labeled handlebodies. We establish an equivalence between the categories $\Tan^G$ and $3\Cob^G$ in Subsection~\ref{S:handle_surgery_functors}.

In Section~\ref{S:Hopf_group-bialgebras}, we introduce the algebraic structure that we will use to define a braided monoidal functor on $\Tan^G$. Namely, we introduce the notion of \textit{Hopf group-bialgebras} in Section~\ref{S:Hopf_group-bialgebras}. We define \textit{ribbon structures} in Subsection~\ref{S:ribbon_structure}, and we introduce the condition of \textit{factorizability} in Subsection~\ref{S:factorizability}. These are the Hopf $G$-bialgebra version of the structure and properties required for the construction of Kerler and Lyubashenko's TQFTs.

In Section~\ref{S:Construction_of_the_functor}, we consider a factorizable Hopf $G$-bialgebra $H$, and we construct a braided monoidal functor $J_H : 3\Cob^G \to \mods{H_{\mathrm{0}}^\oplus}$. In order to do this, we present an algorithm to define the image of a $G$-labeled top tangle. In Subsection~\ref{S:invariance}, we show that the result is invariant under diagrammatic relations in $\Tan^G$ corresponding to $G$-labeled versions of Kirby moves for surgery presentations. This yields a proof of  Theorem~\ref{T:main}. In Subsection~\ref{S:relation_with_Lyubashenko}, we state and prove Corollary~\ref{C:relation_with_KL}, which shows that $J_H$ restricts to the Kerler--Lyubashenko TQFT when labels are all zero.

In Section~\ref{S:quantum_sl2_algebra}, we give an important example of a factorizable ribbon Hopf $G$-bialgebra arising from the quantum group of $\fsl_2$ for $G=\C/2\Z$. The TQFT functor it induces thus depends on cohomology classes with coefficients in the same group appearing in the construction of Blanchet--Costantino--Geer--Patureau. In Subsections~\ref{S:half-divided_basis}, \ref{S:Identities_in_hdb}, and \ref{S:ribbon_factorizability_sl2} we provide bases for state spaces and define ribbon structure, integrals, and cointegrals. Proposition~\ref{P:sl2_is_factorizable} shows that the resulting Hopf $G$-bialgebra is factorizable, and thus can be used to construct a TQFT functor. Finally, in Subsection~\ref{S:J_H_is_homological}, we recall the conjecture relating mapping class groups representations arising from $J_H$ in the case of $\fsl_2$ to homological representations.

\section{Decorated connected framed cobordisms}\label{S:cobordisms}

In this section, we fix an abelian group $G$, and we introduce the category $3\Cob^G$ of \textit{$G$-decorated connected framed cobordisms}. Decorations will be given by cohomology classes with $G$-coefficients, and the definition will generalize the category of connected cobordisms between connected surfaces with connected boundary introduced by Crane--Yetter and Kerler \cite{CY94, K01}.

Let $G$ be an abelian group. The \textit{category $3\Cob^G$ of $G$-decorated connected framed cobordisms} is defined as follows:
\begin{itemize}
 \item Objects of $3\Cob^G$ are triples $\bfSigma = (\varSigma,\vartheta,\calL)$ where
  \begin{itemize}
   \item $\varSigma$ is a connected surface with one boundary component;
   \item $\vartheta \in H^1(\varSigma;G)$ is a cohomology class;
   \item $\calL \subset H_1(\varSigma;\R)$ is a Lagrangian subspace with respect to the intersection form.
  \end{itemize}
 \item Morphisms of $3\Cob^G$ from $\bfSigma = (\varSigma,\vartheta,\calL)$ to $\bfSigma' = (\varSigma',\vartheta',\calL')$ are equivalence classes of triples $\bfM = (M,\omega,\sig)$, where 
  \begin{itemize}
   \item $M$ is a connected $3$-di\-men\-sion\-al cobordism from $\varSigma$ to $\varSigma'$;
   \item $\omega \in H^1(M;G)$ is a cohomology class satisfying $\iota_{-M}^*(\omega) = \vartheta$ and $\iota_{+M}^*(\omega) = \vartheta'$ for the inclusions $\iota_{-M} : \varSigma \hookrightarrow M$ and $\iota_{+M} : \varSigma' \hookrightarrow M$ determined by the boundary identifications;
   \item $\sig$ is an integer, called the \textit{signature defect}.
  \end{itemize}
  Two triples $\bfM = (M,\omega,\sig)$ and $\bfM' = (M',\omega',\sig')$ are \textit{equivalent} if $\sig = \sig'$, and there exists a diffeomorphism $f : M \to M'$ satisfying $f \circ \iota_{\pm M} = \iota_{\pm M'}$ and $f^*(\omega') = \omega$.
 \item The composition 
  \[
   \bfM' \circ \bfM \in 3\Cob^G(\bfSigma,\bfSigma'')
  \]
  of morphisms $\bfM \in 3\Cob^G(\bfSigma,\bfSigma')$ and $\bfM' \in 3\Cob^G(\bfSigma',\bfSigma'')$ is the equivalence class of the triple
  \[
   (M \cup_{\varSigma'} M', \omega \cup_{\vartheta'} \omega', \sig + \sig' - \mu(M_*(\calL),\calL',(M')^*(\calL''))),
  \]
  where:  
  \begin{itemize}
   \item $\omega \cup_{\vartheta'} \omega'$ is the unique cohomology class satisfying
    \[
     \iota_M^*(\omega \cup_{\theta'} \omega') = \omega, \qquad
     \iota_{M'}^*(\omega \cup_{\theta'} \omega') = \omega',
    \]
    for the inclusions $\iota_M : M \hookrightarrow M \cup_{\varSigma'} M'$ and $\iota_{M'} : M' \hookrightarrow M \cup_{\varSigma'} M'$, whose existence and uniqueness is guaranteed by the Mayer--Vietoris sequence
    \begin{center}
     \begin{tikzpicture}
      \node[right] (P0) at (0,0) {$\cdots$};
      \node[right] (P1) at (1.25,0) {$H^0(\varSigma';G)$};
      \node[left] (P2) at (8.75,0) {$H^1(M \cup_{\varSigma'} M';G)$};
      \node (P3) at (10,0) {};
      \node (P4) at (0,-1) {};
      \node[right] (P5) at (1.25,-1) {$H^1(M;G) \oplus H^1(M';G)$};
      \node[left] (P6) at (8.75,-1) {$H^1(\varSigma';G)$};
      \node[left] (P7) at (10,-1) {$\cdots$};
      \draw
      (P0) edge[->] (P1)
      (P1) edge[->] node[above] {\scriptsize $\partial^*$} (P2)
      (P2) edge[-] (P3)
      (P4) edge[->] node[above] {\scriptsize $(\iota_M^*,\iota_{M'}^*)$} (P5)
      (P5) edge[->] node[above] {\scriptsize $\iota_{+M}^* - \iota_{-M'}^*$} (P6)
      (P6) edge[->] (P7);
     \end{tikzpicture} 
    \end{center}
    by remarking that $\partial^*$ is the zero map;
   \item $M_*(\calL)$ and $(M')^*(\calL'')$ are Lagrangian subspaces of $H_1(\varSigma';\R)$ obtained by pushing forward $\calL$ through $M$ and by pulling back $\calL''$ through $M'$, respectively, and $\mu$ is the Maslov index, see \cite[Section~2.2]{BD21} for more details.
  \end{itemize}
 \item The identity
  \[
   \id_{\bfSigma} \in 3\Cob^G(\bfSigma,\bfSigma)
  \]
  of an object $\bfSigma \in 3\Cob^G$ is given by
  \[
   (\varSigma \times [0,1], \vartheta \times [0,1], 0),
  \]
  where $\vartheta \times [0,1] = \pi^*(\vartheta)$ for the projection $\pi : \varSigma \times [0,1] \to \varSigma$.
\end{itemize}

\begin{remark}\label{R:G=0_Cob}
 When $G = 0$, then $3\Cob^G$ is naturally isomorphic to Crane and Yetter's category $3\Cob$ of connected framed cobordisms between connected surfaces with connected boundary \cite{CY94, K01}.
\end{remark}

\subsection{Labeled top tangles in labeled handlebodies}\label{S:top_tangles}

In this section, we introduce the category $\Tan^G$ of \textit{$G$-labeled top tangles in $G$-labeled handlebodies}, which, as we will show, is equivalent to the category $3\Cob^G$. By contrast with $3\Cob^G$, though, $\Tan^G$ is naturally a braided monoidal category. It is a generalization of the category of bottom tangles in handlebodies introduced by Habiro \cite{H05}.

For every non-negative integer $g \geqs 0$, we specify a connected $3$-di\-men\-sion\-al handlebody $M_g \subset \R^3$ of genus $g$, obtained by attaching $g$ copies $B_1, \ldots, B_g$ of the $3$-di\-men\-sion\-al $1$-han\-dle $B = D^1 \times D^2$ to the bottom face $[0,1]^{\times 2} \times \{ 0 \}$ of the cube $[0,1]^{\times 3}$. We represent graphically $M_g$ through the projection to $\R \times \{ 0 \} \times \R$ as
\begin{align*}
 \pic{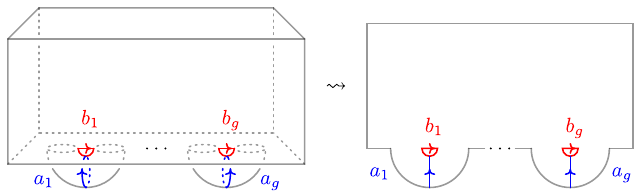}
\end{align*}
We denote by $\varSigma_g$ the connected surface of genus $g$ with one boundary component appearing as the bottom face of $M_g$, which is obtained from the square $[0,1]^{\times 2} \times \{ 0 \}$ by performing $1$-surgery along $g$ pairs of discs, and we denote by $\{ a_1,b_1,\ldots,a_g,b_g \}$ the set of curves represented above, which form a standard basis of $H_1(\varSigma_g)$.

A \textit{$G$-labeled handlebody} is a standard connected handlebody $M_g$ equipped with a $G$-labeling of its $1$-handles, given by a label $\beta_i \in G$ attached to its $1$-handle $B_i$ for every integer $1 \leqs i \leqs g$.

A \textit{$G$-labeled top tangle in an $G$-labeled handlebody}, sometimes simply called a \textit{top tangle}, is an oriented framed $G$-labeled tangle $T = A'_1 \cup \ldots A'_{g'} \cup C_1 \cup \ldots \cup C_h$ inside a connected $G$-labeled handlebody $M_g$ satisfying the following list of conditions:
\begin{itemize}
 \item The set of boundary points of $T$ is composed of $2g'$ points uniformly distributed on the top line $[0,1] \times \{ \frac 12 \} \times \{ 1 \} \subset M_g$;
 \item For every $1 \leqs j \leqs g'$, an arc component $A'_j$ of $T$ joins the $(2j)$th boundary point to the $(2j-1)$th one;
 \item For every $1 \leqs \ell \leqs h$, we have
  \begin{equation}
   \sum_{i=1}^g \lk(a_i,C_\ell) \beta_i + \sum_{j=1}^{g'} \lk(\hat{A}'_j,C_\ell) \alpha'_j + \sum_{k=1}^h \lk(C_k,C_\ell) \gamma_k = 0, \label{E:middle}
  \end{equation}
  where $\beta_i \in G$ is the label of the $1$-handle $B_i$ of $M_g$, where $\alpha'_j \in G$ is the label of the arc component $A'_j$ of $T$, whose plat closure is denoted $\hat{A}'_j$, where $\gamma_k \in G$ is the label of the circle component $C_k$ of $T$, and where $\lk(C_\ell,C_\ell)$ stands for the framing of $C_\ell$.
\end{itemize}
The source $s(T)$ of a top tangle $T$ is
\[
 s(T) = (\myuline{\alpha},\myuline{\beta}) = (\alpha_1,\beta_1,\ldots,\alpha_g,\beta_g) \in G^{\times 2g},
\]
where
\begin{align}
 \alpha_\ell &:= \sum_{j=1}^{g'} \lk(\hat{A}'_j,a_\ell) \alpha'_j + \sum_{k=1}^h \lk(C_k,a_\ell) \gamma_k \label{E:source}
\end{align}
and $\beta_\ell$ is the label of the $1$-handle $B_\ell$ of $M_g$ for every integer $1 \leqs \ell \leqs g$. The target $t(T)$ of a top tangle $T$ is
\[
 t(T) = (\myuline{\alpha'},\myuline{\beta'}) = (\alpha'_1,\beta'_1,\ldots,\alpha'_{g'},\beta'_{g'}) \in G^{\times 2g'},
\]
where $\alpha'_\ell$ is the label of the arc component $A'_\ell$ of $T$ and
\begin{align}
 \beta'_\ell &:= \sum_{i=1}^g \lk(a_i,\hat{A}'_\ell) \beta_i + \sum_{j=1}^{g'} \lk(\hat{A}'_j,\hat{A}'_\ell) \alpha'_j + \sum_{k=1}^h \lk(C_k,\hat{A}'_\ell) \gamma_k \label{E:target}
\end{align}
for every integer $1 \leqs \ell \leqs g'$.

Here is an example of a top tangle, together with its projection:
\begin{align*}
 \pic{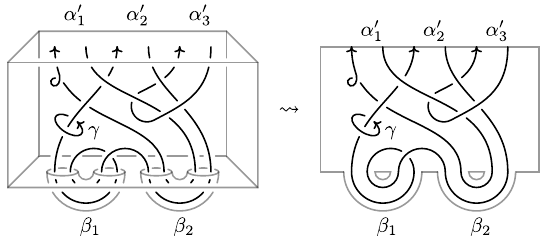}
\end{align*}
In this example, the source is $(\alpha_1,\beta_1,\alpha_2,\beta_2)$ with
\[
 \alpha_1 = 2 \alpha'_2, \qquad
 \alpha_2 = \alpha'_1+\alpha'_2,
\]
the target is $(\alpha'_1,\beta'_1,\alpha'_2,\beta'_2,\alpha'_3,\beta'_3)$ with
\[
 \beta'_1 = \alpha'_2 + \alpha'_3, \qquad
 \beta'_2 = \alpha'_1 + \alpha'_2 + \alpha'_3 + \gamma, \qquad
 \beta'_3 = \alpha'_1 + \alpha'_2,
\] 
and $\alpha'_2 = 0$. Notice that, up to isotopy, we can always represent top tangles using \textit{regular diagrams}, that are diagrams in which every $1$-handle $D^1 \times D^2$ intersects the tangle $T$ in $D^1 \times P \subset D^1 \times D^2$, where $P \subset D^1$ is a finite set of points that remain distinct under the projection.

\begin{remark}
 We will see later that $G$-labeled top tangles correspond to \dmnsnl{3} cobordisms decorated with cohomology classes. Under this correspondence, $G$-labelings encode the evaluation of a cohomology class with coefficients in $G$ against \dmnsnl{1} homology classes determined by meridians of the tangle and by longitudes of the handles.
\end{remark}

We consider top tangles up to \textit{framed $G$-Kirby moves} of the following type:
\begin{align*}
 \pic{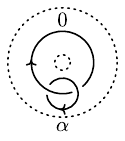} &\leftrightsquigarrow \pic{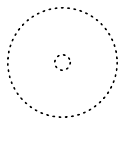} \tag{K$1$}\label{E:K1} \\
 \pic{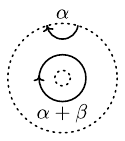} &\leftrightsquigarrow \pic{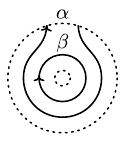} \tag{K$2$}\label{E:K2} \\
 \pic{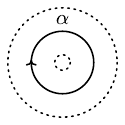} &\leftrightsquigarrow \pic{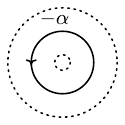} \tag{K$3$}\label{E:K3}
\end{align*}
These operations can be performed inside any solid torus $S^1 \times D^2$ arbitrarily embedded into $M_g$.

The \textit{category $\Tan^G$ of $G$-labeled top tangles in $G$-labeled handlebodies} is the category defined as follows:
\begin{itemize}
 \item Objects of $\Tan^G$ are finite sequences $(\myuline{\alpha},\myuline{\beta}) = (\alpha_1,\beta_1,\ldots,\alpha_g,\beta_g) \in G^{\times 2g}$ for $g \geqs 0$.
 \item Morphisms of $\Tan^G$ from $(\myuline{\alpha},\myuline{\beta}) \in G^{\times 2g}$ to $(\myuline{\alpha'},\myuline{\beta'}) \in G^{\times 2g'}$ are isotopy classes of top tangles $T$ with source $(\myuline{\alpha},\myuline{\beta})$ and target $(\myuline{\alpha'},\myuline{\beta'})$.
 \item The composition 
  \[
   T' \circ T \in \Tan^G((\myuline{\alpha},\myuline{\beta}),(\myuline{\alpha''},\myuline{\beta''}))
  \]
  of morphisms $T \in \Tan^G((\myuline{\alpha},\myuline{\beta}),(\myuline{\alpha'},\myuline{\beta'}))$, $T' \in \Tan^G((\myuline{\alpha'},\myuline{\beta'}),(\myuline{\alpha''},\myuline{\beta''}))$ is obtained by considering an open tubular neighborhood $N(\tilde{T})$ in $M_g$ of the subtangle $\tilde{T} = A_1 \cup \ldots \cup A_{g'}$ of $T$ composed of all arc components, by gluing vertically its complement $H_g \smallsetminus N(\tilde{T})$ to $M_{g'}$, identifying the top base of $M_g \smallsetminus N(\tilde{T})$ with the bottom base of $M_{g'}$ as prescribed by the top tangle $\tilde{T}$, and then by shrinking the result into $M_g$.
 \item The identity $\id_{(\myuline{\alpha},\myuline{\beta})} \in \Tan^G((\myuline{\alpha},\myuline{\beta}),(\myuline{\alpha},\myuline{\beta}))$ of an object $(\myuline{\alpha},\myuline{\beta}) \in \Tan^G$ is given by
  \[
   \id_{(\myuline{\alpha},\myuline{\beta})} := \pic{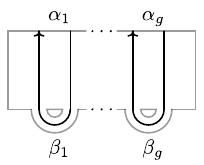}
  \]
\end{itemize}
Here is an example of a composition of top tangles:
\begin{align*}
 \pic{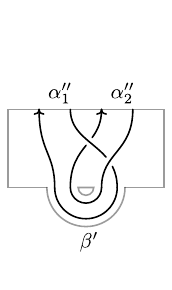} \circ \pic{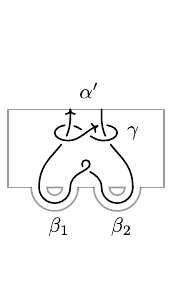} = \pic{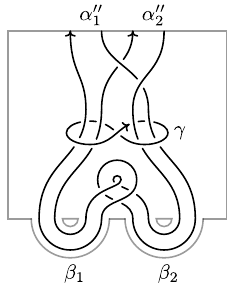}
\end{align*}
In this example, we are composing a top tangle whose source $(\alpha_1,\beta_1,\alpha_2,\beta_2)$ and target $(\alpha',\beta')$ satisfy
\[
 \alpha_1 = \alpha_2 = \alpha', \qquad
 \beta' = \beta_1 + \beta_2 - \alpha' + 2 \gamma, \qquad
 2 \alpha' + \gamma = 0
\]
with a top tangle whose source $(\alpha',\beta')$ and target $(\alpha''_1,\beta''_1,\alpha''_2,\beta''_2)$ satisfy
\[
 \alpha' = \alpha''_1 + \alpha''_2, \qquad
 \beta''_1 = \beta' + \alpha''_2, \qquad
 \beta''_2 = \beta' + \alpha''_1.
\]

The category $\Tan^G$ can be given the structure of a braided monoidal category:
\begin{itemize}
 \item The tensor product
  \[
   (\myuline{\alpha},\myuline{\beta}) \bcs (\myuline{\alpha'},\myuline{\beta'}) \in \Tan^G
  \]
  of objects $(\myuline{\alpha},\myuline{\beta}), (\myuline{\alpha'},\myuline{\beta'}) \in \Tan^G$ is given by their concatenation.
 \item The tensor product
  \[
   T \bcs T' \in \Tan^G((\myuline{\alpha},\myuline{\beta}) \bcs (\myuline{\alpha'},\myuline{\beta'}),(\myuline{\alpha''},\myuline{\beta''}) \bcs (\myuline{\alpha'''},\myuline{\beta'''}))
  \]
  of morphisms $T \in \Tan^G((\myuline{\alpha},\myuline{\beta}),(\myuline{\alpha''},\myuline{\beta''}))$, $T' \in \Tan^G((\myuline{\alpha'},\myuline{\beta'}),(\myuline{\alpha'''},\myuline{\beta'''}))$ is obtained by gluing horizontally $M_g$ to $M_{g'}$, identifying the right side of $M_g$ with the left side of $M_{g'}$ as prescribed by the identity map, and then by shrinking the result into $M_{g+g'}$.
 \item The tensor unit is the empty sequence $\varnothing \in \Tan^G$.
 \item The braiding $c_{(\myuline{\alpha},\myuline{\beta}),(\myuline{\alpha'},\myuline{\beta'})} \in \Tan^G((\myuline{\alpha},\myuline{\beta}) \bcs (\myuline{\alpha'},\myuline{\beta'}),(\myuline{\alpha'},\myuline{\beta'}) \bcs (\myuline{\alpha},\myuline{\beta}))$ of objects $(\myuline{\alpha},\myuline{\beta}), (\myuline{\alpha'},\myuline{\beta'}) \in \Tan^G$ is given by
  \[
   c_{(\myuline{\alpha},\myuline{\beta}),(\myuline{\alpha'},\myuline{\beta'})} := \pic{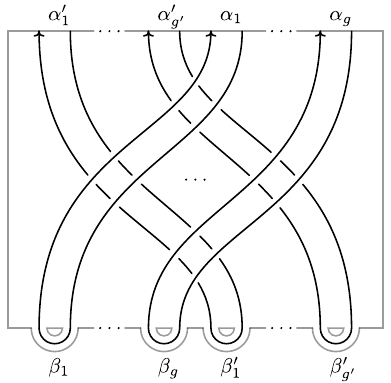}
  \]
\end{itemize}
Here is an example of a tensor product of top tangles:
\begin{align*}
 \pic[-10]{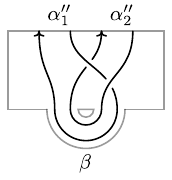} \bcs \pic[-10]{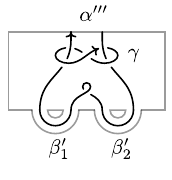} = \pic[-10]{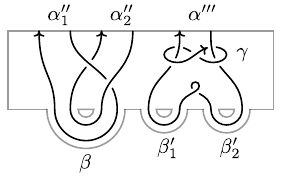}
\end{align*}
In this example, we are tensoring a top tangle whose source $(\alpha,\beta)$ and target $(\alpha''_1,\beta''_1,\alpha''_2,\beta''_2)$ satisfy
\[
 \alpha = \alpha''_1 + \alpha''_2, \qquad
 \beta''_1 = \beta + \alpha''_2, \qquad
 \beta''_2 = \beta + \alpha''_1
\]
with a top tangle whose source $(\alpha'_1,\beta'_1,\alpha'_2,\beta'_2)$ and target $(\alpha''',\beta''')$ satisfy
\[
 \alpha'_1 = \alpha'_2 = \alpha''', \qquad
 \beta''' = \beta'_1 + \beta'_2 - \alpha''' + 2 \gamma, \qquad
 2 \alpha''' + \gamma = 0.
\]

\begin{remark}\label{R:G=0_Tan}
 When $G = 0$, then $\Tan^G$ is naturally isomorphic to the category $\Tan$ of framed top tangles in handlebodies \cite{BD21, BBDP23}. Indeed, notice that, even though morphisms in $\Tan$ are given by unoriented tangles, there is always a unique orientation for arc components that is compatible with our convention for $\Tan^0$, and $0$-labeled closed components can be given any arbitrary orientation thanks to move~\eqref{E:K3}.
\end{remark}

\subsection{Equivalence of categories}\label{S:handle_surgery_functors}

We consider now the \textit{surgery} functor 
\[
 \chi : \Tan^G \to 3\Cob^G
\]
sending every object $(\myuline{\alpha},\myuline{\beta}) = (\alpha_1,\beta_1,\ldots,\alpha_g,\beta_g) \in G^{\times 2g}$ of $\Tan^G$ to the object
\[
 \bfSigma_{(\myuline{\alpha},\myuline{\beta})} := (\varSigma_g,\vartheta_{(\myuline{\alpha},\myuline{\beta})},\calL_g)
\]
of $3\Cob^G$, where 
\begin{itemize}
 \item $\vartheta_{(\myuline{\alpha},\myuline{\beta})} \in H^1(\varSigma_g;G)$ is the cohomology class determined by 
\[
 \vartheta_{(\myuline{\alpha},\myuline{\beta})}(a_i) = \alpha_i, \qquad
 \vartheta_{(\myuline{\alpha},\myuline{\beta})}(b_i) = \beta_i
\]
for every integer $1 \leqs i \leqs g$,
 \item $\calL_g \subset H_1(\varSigma_g;\R)$ is the Lagrangian subspace given by $\ker \iota_{-M_g}^*$ for the inclusion $\iota_{-M_g} : \varSigma_g \hookrightarrow M_g$,
\end{itemize}
and sending every morphism $T$ in $\Tan^G((\myuline{\alpha},\myuline{\beta}),(\myuline{\alpha'},\myuline{\beta'}))$ to the morphism 
\[ 
 \bfM(T) := (M(T),\omega_T,\sigma(T \smallsetminus \tilde{T}))
\]
in $3\Cob(\bfSigma_{(\myuline{\alpha},\myuline{\beta})},\bfSigma_{(\myuline{\alpha'},\myuline{\beta'})})$, where 
\begin{itemize}
 \item $M(T)$ is the cobordism obtained from $M_g$ by carving out an open tubular neighborhood $N(\tilde{T})$ in $H_m$ of the open subtangle $\tilde{T} = A'_1 \cup \ldots \cup A'_{g'}$ of $T$ composed of all arc components, and by performing $2$-surgery along the closed subtangle $T \smallsetminus \tilde{T} = C_1 \cup \ldots \cup C_h$ composed of all circle components,
 \item $\omega_T \in H^1(M(T);G)$ is the cohomology class determined by 
  \[
   \omega_T(b_i) = \beta_i, \qquad
   \omega_T(a'_j) = \alpha'_j, \qquad
   \omega_T(c_k) = \gamma_k
  \]
  for all integers $1 \leqs i \leqs g$, $1 \leqs j \leqs g'$, $1 \leqs k \leqs h$, where $a'_j$ and $c_k$ denote positive meridians of the arc component $A'_j$ and of the circle component $C_k$ of $T$, respectively, and where $\alpha'_j, \gamma_k \in G$ denote the respective labels;
 \item $\sigma(T \smallsetminus \tilde{T})$ is the signature of the linking matrix of $(T \smallsetminus \tilde{T}) \subset M_g \subset \R^3$.
\end{itemize}

\begin{proposition}\label{P:surgery_equivalence}
 The surgery functor $\chi : \Tan^G \to 3\Cob^G$ is an equivalence.
\end{proposition}

\begin{proof}
 Since, for every $g \geqs 0$, the set of curves $\{ a_1,b_1,\ldots,a_g,b_g \}$ provides a basis of $H_1(\varSigma_g)$, then $\chi$ is essentially surjective on objects, as a direct consequence of the classification of surfaces. Furthermore, every cobordism $M$ from $\varSigma_g$ to $\varSigma_{g'}$ can be obtained from 
 \[
  M_0 := M \left( \pic{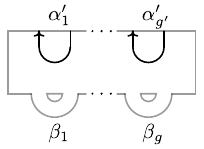} \right)
 \]
 by $2$-surgery along some framed link $L$, and every cohomology class $\omega \in H^1(M;G)$ determines a label $\omega(c_k) \in G$ for every positive meridian $c_k$ of every component $C_k$ of $L$. These labels have to satisfy Equation~\eqref{E:middle} because the longitude of $C_k$ determined by the framing is null-homologous in $M_0(L)$. This shows that $\chi$ is full. Finally, two $G$-labeled top tangles $T$ and $T'$ in $\Tan^G$ determine diffeomorphic cobordisms under $\chi$ if and only if they are related by a finite sequence of framed $G$-Kirby moves~\eqref{E:K1}--\eqref{E:K3}, as follows from \cite[Theorem~1]{R97}. This shows $\chi$ is faithful.
\end{proof}

\begin{remark}\label{R:generators}
 Thanks to Remarks~\ref{R:G=0_Cob} and \ref{R:G=0_Tan}, and to \cite[Theorem~2.2]{K01}, if $G = 0$, then $\Tan \cong \Tan^0$ is generated by the finite list of morphisms described in \cite[Equations~(46)--(50)]{K01}, compare with \cite[Figure~3.4.3]{BBDP23}. In fact, in this case, a complete algebraic presentation of $\Tan$ is given in \cite[Corollary~B \& Theorem~C]{BBDP23}. It is clear then that, for an arbitrary abelian group $G$, a list of generators for $\Tan^G$ is obtained by considering all possible $G$-labelings of generators of $\Tan$. In other words, the list of morphisms represented in Figure~\ref{F:generators}, for all $\alpha, \beta, \gamma \in G$, provides generators for $\Tan^G$.
\end{remark}

\begin{figure}[t]
 \includegraphics{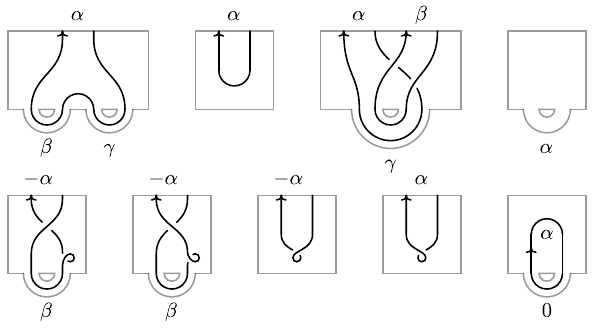}
 \caption{Generators of $\Tan^G$, with $\alpha, \beta, \gamma \in G$.}\label{F:generators}
\end{figure}

\section{Hopf group-bialgebras}\label{S:Hopf_group-bialgebras}

In this section, we introduce the notion of Hopf group-bialgebras, which generalize Ohtsuki's \textit{colored Hopf algebras} \cite{O93}, which are a special case of \textit{Hopf group-coalgebras}, as introduced by Turaev \cite{T10} and Virelizier \cite{V00}.

Let $G$ be an abelian group. A \textit{Hopf $G$-bialgebra} is a family $H = \{ H_\alpha^\beta \mid \alpha,\beta \in G \}$ of vector spaces over a field $\Bbbk$ equipped with
\begin{itemize}
 \item a \textit{product} $\mu = \{ \mu_\alpha^{\beta,\gamma} : H_\alpha^\beta \otimes H_\alpha^\gamma \to H_\alpha^{\beta+\gamma} \mid \alpha,\beta,\gamma \in G \}$,
 \item a \textit{unit} $\eta = \{ \eta_\alpha^0 : \Bbbk \to H_\alpha^0 \mid \alpha \in G \}$,
 \item a \textit{coproduct} $\Delta = \{ \Delta_{\alpha,\beta}^\gamma : H_{\alpha+\beta}^\gamma \to H_\alpha^\gamma \otimes H_\beta^\gamma \mid \alpha,\beta,\gamma \in G \}$,
 \item a \textit{counit} $\varepsilon = \{ \varepsilon_0^\alpha : H_0^\alpha \to \Bbbk \mid \alpha \in G \}$,
 \item an \textit{antipode} $S = \{ S_\alpha^\beta : H_\alpha^\beta \to H_{-\alpha}^{-\beta} \mid \alpha,\beta \in G \}$.
\end{itemize}
For all $\alpha,\beta,\gamma,\delta \in G$, these data satisfy
\begin{gather*}
 \mu_\alpha^{\beta+\gamma,\delta} \circ (\mu_\alpha^{\beta,\gamma} \otimes \id_{H_\alpha^\delta}) = \mu_\alpha^{\beta,\gamma+\delta} \circ (\id_{H_\alpha^\beta} \otimes \mu_\alpha^{\gamma,\delta}), \tag{H$1$}\label{E:ass} \\
 \mu_\alpha^{0,\beta} \circ (\eta_\alpha^0 \otimes \id_{H_\alpha^\beta}) = \id_{H_\alpha^\beta} = \mu_\alpha^{\beta,0} \circ (\id_{H_\alpha^\beta} \otimes \eta_\alpha^0), \tag{H$2$}\label{E:unit} \\
 (\Delta_{\alpha,\beta}^\delta \otimes \id_{H_{\gamma}^\delta}) \circ \Delta_{\alpha+\beta,\gamma}^\delta = (\id_{H_\alpha^\delta} \otimes \Delta_{\beta,\gamma}^\delta) \circ \Delta_{\alpha,\beta+\gamma}^\delta, \tag{H$3$}\label{E:coass} \\
 (\varepsilon_0^\beta \otimes \id_{H_\alpha^\beta}) \circ \Delta_{0,\alpha}^\beta = \id_{H_\alpha^\beta} = (\id_{H_\alpha^\beta} \otimes \varepsilon_0^\beta) \circ \Delta_{\alpha,0}^\beta, \tag{H$4$}\label{E:coun} \\
 \Delta_{\alpha,\beta}^{\gamma+\delta} \circ \mu_{\alpha+\beta}^{\gamma,\delta} = (\mu_\alpha^{\gamma,\delta} \otimes \mu_\beta^{\gamma,\delta}) \circ (\id_{H_\alpha^\gamma} \otimes \tau_{H_\beta^\gamma,H_\alpha^\delta} \otimes \id_{H_\beta^\delta}) \circ (\Delta_{\alpha,\beta}^\gamma \otimes \Delta_{\alpha,\beta}^\delta), \tag{H$5$}\label{E:prod-coprod} \\
 \varepsilon_0^{\alpha+\beta} \circ \mu_0^{\alpha,\beta} = \varepsilon_0^\alpha \otimes \varepsilon_0^\beta, \tag{H$6$}\label{E:prod-coun} \\
 \Delta_{\alpha,\beta}^0 \circ \eta_{\alpha+\beta}^0 = \eta_\alpha^0 \otimes \eta_\beta^0, \tag{H$7$}\label{E:unit-coprod} \\
 \varepsilon_0^0 \circ \eta_0^0 = \id_\Bbbk, \tag{H$8$}\label{E:unit-coun} \\
 \mu_\alpha^{-\beta,\beta} \circ (S_{-\alpha}^\beta \otimes \id_{H_\alpha^\beta}) \circ \Delta_{-\alpha,\alpha}^\beta = \eta_\alpha^0 \circ \varepsilon_0^\beta = \mu_\alpha^{\beta,-\beta} \circ (\id_{H_\alpha^\beta} \otimes S_{-\alpha}^\beta) \circ \Delta_{\alpha,-\alpha}^\beta, \tag{H$9$}\label{E:antip}
\end{gather*}
where $\tau_{H_\beta^\gamma,H_\alpha^\delta} : H_\beta^\gamma \otimes H_\alpha^\delta \to H_\alpha^\delta \otimes H_\beta^\gamma$ is the standard transposition. We will use the shorthand notation 
\[
 \eta_\alpha = \eta_\alpha^0, \qquad 1_\alpha = \eta_\alpha(1), \qquad \varepsilon^\alpha = \varepsilon_0^\alpha
\]
for every $\alpha \in G$,
\[
 xy = \mu_\alpha^{\beta,\gamma}(x \otimes y)
\]
for all $x \in H_\alpha^\beta$, $y \in H_\alpha^\gamma$, and
\begin{align*}
 z_{(1,\alpha_1)} \otimes z_{(2,\alpha_2)} 
 &= \Delta_{\alpha_1,\alpha_2}^\beta(z) & \mbox{ if } n=2 \\
 z_{(1,\alpha_1)} \otimes \ldots \otimes z_{(n,\alpha_n)}
 &= \mathrlap{z_{(1,\alpha_1)} \otimes \ldots \otimes z_{(n-2,\alpha_{n-2})} \otimes \Delta_{\alpha_{n-1},\alpha_n}^\beta(z_{(n-1,\alpha_{n-1}+\alpha_n)})} & \\
 &= \Delta_{\alpha_1,\alpha_2}^\beta(z_{(1,\alpha_1+\alpha_2)}) \otimes z_{(2,\alpha_3)} \otimes \ldots \otimes z_{(n-1,\alpha_n)} & \mbox{ if } n>2 
\end{align*}
for every $z \in H_{\alpha_1+\ldots+\alpha_n}^\beta$.

\begin{remark}\label{R:Hopf_G-coalgebra}
 If $H = \{ H_\alpha^\beta \mid \alpha,\beta \in G \}$ is a Hopf $G$-bialgebra, then $H_\alpha^0$ is an algebra and $H_0^\alpha$ is a coalgebra for every $\alpha \in G$. In particular, $H_0^0$ is a Hopf algebra. Furthermore, 
 \[
  H^\oplus = \left\{ H_\alpha^\oplus = \bigoplus_{\beta \in G} H_\alpha^\beta \Biggm\vert \alpha \in G \right\}
 \]
 is a Hopf $G$-coalgebra, and $H_\alpha^\oplus$ is an algebra for every $\alpha \in G$. We will write
 \[
  (H_\alpha^\beta)^\times = (H_\alpha^\oplus)^\times \cap H_\alpha^\beta, \qquad
  Z(H_\alpha^\beta) = Z(H_\alpha^\oplus) \cap H_\alpha^\beta
 \]
 for all $\alpha, \beta \in G$ to denote the intersection of $H_\alpha^\beta$ with the set of invertible elements and with the center of $H_\alpha^\oplus$, respectively.
\end{remark}

\begin{proposition}
 For all $\alpha, \beta, \gamma \in G$, the antipode satisfies
 \begin{gather}
  S_\alpha^{\beta+\gamma}(xy) = S_\alpha^\gamma(y) S_\alpha^\beta(x), \tag{H$10$}\label{E:antip-prod} \\
  S_\alpha^0(1_\alpha) = 1_{-\alpha} \tag{H$11$}\label{E:antip-unit} \\
  (S_{\alpha+\beta}^\gamma(z))_{(1,-\alpha)} \otimes (S_{\alpha+\beta}^\gamma(z))_{(2,-\beta)} = S_\alpha^\gamma(z_{(2,\alpha)}) \otimes S_\beta^\gamma(z_{(1,\beta)}), \tag{H$12$}\label{E:antip-coprod} \\
  \varepsilon^\alpha(S_0^\alpha(w)) = \varepsilon^{-\alpha}(w). \tag{H$13$}\label{E:antip-coun}
 \end{gather}
 for all $x \in H_\alpha^\gamma$, $y \in H_\beta^\gamma$, $z \in H_{\alpha+\beta}^\gamma$, $w \in H_0^\alpha$.
\end{proposition}

\begin{proof}
 Equations~\eqref{E:antip-prod}--\eqref{E:antip-coun} are the Hopf $G$-bialgebra version of \cite[Theorem~III.3.4.(a)]{K95}, and the reader can directly adapt the proof from there. Notice that the Hopf $G$-coalgebra version of these properties is \cite[Lemma~1.1]{V00}.
\end{proof}

\subsection{Ribbon structures}\label{S:ribbon_structure}

If $G$ is a ring such that $2 \in G^\times$, which justifies the notation $\frac 12 = 2^{-1} \in G^\times$, then a \textit{ribbon Hopf $G$-bialgebra} is a Hopf $G$-bialgebra $H = \{ H_\alpha^\beta \mid \alpha,\beta \in G \}$ equipped with
\begin{itemize}
 \item an \textit{R-matrix} $R = \{ R_{\alpha,\beta}^{\frac{\beta}{2},\frac{\alpha}{2}} = \displaystyle \sum_i (R'_i)_\alpha^{\frac{\beta}{2}} \otimes (R''_i)_\beta^{\frac{\alpha}{2}} \in (H_\alpha^{\frac{\beta}{2}} \otimes H_\beta^{\frac{\alpha}{2}})^\times \mid \alpha,\beta \in G \}$,
 \item a \textit{ribbon element} $v = \{ v_\alpha^{-\alpha} \in Z(H_\alpha^{-\alpha})^\times \mid \alpha \in G \}$.
 \end{itemize}
For all $\alpha,\beta,\gamma \in G$, these data satisfy:
\begin{gather*}
 \sum_i (R'_i)_\alpha^{\frac{\beta}{2}} x_{(1,\alpha)} \otimes (R''_i)_\beta^{\frac{\alpha}{2}} x_{(2,\beta)} = \sum_i x_{(2,\alpha)} (R'_i)_\alpha^{\frac{\beta}{2}} \otimes x_{(1,\beta)} (R''_i)_\beta^{\frac{\alpha}{2}}, \tag{R$1$}\label{E:R-intertw-coprod} \\
 \sum_i (R'_i)_\alpha^{\frac{\beta+\gamma}{2}} \otimes \Delta_{\beta,\gamma}^{\frac{\alpha}{2}}((R''_i)_{\beta+\gamma}^{\frac{\alpha}{2}}) = \sum_{j,k} (R'_j)_\alpha^{\frac{\gamma}{2}} (R'_k)_\alpha^{\frac{\beta}{2}} \otimes (R''_k)_\beta^{\frac{\alpha}{2}} \otimes (R''_j)_\gamma^{\frac{\alpha}{2}}, \tag{R$2$}\label{E:R-coprod-left} \\
 \sum_i \Delta_{\alpha,\beta}^{\frac{\gamma}{2}}((R'_i)_{\alpha+\beta}^{\frac{\gamma}{2}}) \otimes (R''_i)_\gamma^{\frac{\alpha+\beta}{2}} = \sum_{j,k} (R'_j)_\alpha^{\frac{\gamma}{2}} \otimes (R'_k)_\beta^{\frac{\gamma}{2}} \otimes (R''_j)_\gamma^{\frac{\alpha}{2}} (R''_k)_\gamma^{\frac{\beta}{2}}, \tag{R$3$}\label{E:R-coprod-right} \\
 (v_\alpha^{-\alpha})^2 = u_\alpha^{-\alpha} S_{-\alpha}^\alpha(u_{-\alpha}^\alpha), \tag{R$4$}\label{E:square-rib} \\
 \Delta_{\alpha,\beta}^{-(\alpha+\beta)}(v_{\alpha+\beta}^{-(\alpha+\beta)}) \hspace*{180pt} \\ 
 = \sum_{i,j} (v_\alpha^{-\alpha} S_{-\alpha}^{\frac{\beta}{2}}((R'_i)_{-\alpha}^{\frac{\beta}{2}}) (R''_j)_\alpha^{-\frac{\beta}{2}} \otimes v_\beta^{-\beta} (R''_i)_\beta^{-\frac{\alpha}{2}} S_{-\beta}^{\frac{\alpha}{2}}((R'_j)_{-\beta}^{\frac{\alpha}{2}})), \tag{R$5$}\label{E:coprod-rib} \\
 \varepsilon^0(v_0^0) = 1, \tag{R$6$}\label{E:coun-rib} \\
 S_\alpha^{-\alpha}(v_\alpha^{-\alpha}) = v_{-\alpha}^\alpha \tag{R$7$}\label{E:antip-rib}
\end{gather*}
for all $x \in H_{\alpha+\beta}^\gamma$ and $y \in H_\alpha^\beta$, and for the \textit{Drinfeld element}
\[
 u_\alpha^{-\alpha} := \sum_i S_{-\alpha}^{\frac{\alpha}{2}}((R''_i)_{-\alpha}^{\frac{\alpha}{2}}) (R'_i)_\alpha^{-\frac{\alpha}{2}}.
\]
The latter determines a \textit{pivotal element}
\[
 g = \{ g_\alpha^0 \in H_\alpha^0 \mid \alpha \in G \}
\]
defined by
\[
 g_\alpha^0 := u_\alpha^{-\alpha} (v^{-1})_\alpha^\alpha.
\]
We will adopt Einstein's notation by suppressing sums, and also use the shorthand notation $g_\alpha = g_\alpha^0$ for every $\alpha \in G$.

\begin{remark}
 If $H = \{ H_\alpha^\beta \mid \alpha,\beta \in G \}$ is a ribbon Hopf $G$-bialgebra, then:
 \begin{enumerate}
  \item Using the notation of Remark~\ref{R:Hopf_G-coalgebra}, $H^\oplus = \left\{ H_\alpha^\oplus \mid \alpha \in G \right\}$ is a ribbon Hopf $G$-coalgebra\footnote{Notice that our ribbon element $v$ corresponds to the inverse of Virelizier's $\vartheta$.} in the sense of \cite[Section~6.4]{V00}, and in particular $H_0^\oplus$ is a ribbon Hopf algebra;
  \item The antipode $S$ of a ribbon Hopf $G$-bialgebra is always invertible in the sense that $S_\alpha^\beta : H_\alpha^\beta \to H_{-\alpha}^{-\beta}$ is a linear isomorphism for all $\alpha, \beta \in G$, as a direct consequence of \cite[Lemma~6.5.(c)]{V00}.
 \end{enumerate}
\end{remark}

\begin{proposition}\label{P:R-matrices_pivotal_elements}
 For all $\alpha,\beta,\gamma \in G$, the R-matrix and the pivotal element satisfy
 \begin{gather}
  \sum_{i,j,k} (R'_i)_\alpha^{\frac{\beta}{2}} (R'_j)_\alpha^{\frac{\gamma}{2}} \otimes (R''_i)_\beta^{\frac{\alpha}{2}} (R'_k)_\beta^{\frac{\gamma}{2}} \otimes (R''_j)_\gamma^{\frac{\alpha}{2}} (R''_k)_\gamma^{\frac{\beta}{2}} \nonumber \\*
  \hspace*{\parindent} = \sum_{i,j,k} (R'_j)_\alpha^{\frac{\gamma}{2}} (R'_k)_\alpha^{\frac{\beta}{2}} \otimes (R'_i)_\beta^{\frac{\gamma}{2}} (R''_k)_\beta^{\frac{\alpha}{2}} \otimes (R''_i)_\gamma^{\frac{\beta}{2}} (R''_j)_\gamma^{\frac{\alpha}{2}}, \tag{R$8$}\label{E:YB} \\
  \sum_i \varepsilon^{\frac{\alpha}{2}} \left( (R''_i)_0^{\frac{\alpha}{2}} \right) (R'_i)_\alpha^0
  = 1_\alpha
  = \sum_i \varepsilon^{\frac{\alpha}{2}} \left( (R'_i)_0^{\frac{\alpha}{2}} \right) (R''_i)_\alpha^0, \tag{R$9$}\label{E:coun-R} \\
   \left( R_{\alpha,\beta}^{\frac{\beta}{2},\frac{\alpha}{2}} \right)^{-1} 
   = \sum_i S_{-\alpha}^{\frac{\beta}{2}} \left( (R'_i)_{-\alpha}^{\frac{\beta}{2}} \right) \otimes (R''_i)_\beta^{-\frac{\alpha}{2}} \nonumber \\*
   \hspace*{65pt}= \sum_i (R'_i)_\alpha^{-\frac{\beta}{2}} \otimes (S_\beta^{-\frac{\alpha}{2}})^{-1} \left( (R''_i)_{-\beta}^{\frac{\alpha}{2}} \right), \tag{R$10$}\label{E:inv-R} \\
  \sum_i S_\alpha^{\frac{\beta}{2}} \left( (R'_i)_\alpha^{\frac{\beta}{2}} \right) \otimes S_\beta^{\frac{\alpha}{2}} \left( (R''_i)_\beta^{\frac{\alpha}{2}} \right) 
  = R_{-\alpha,-\beta}^{-\frac{\beta}{2},-\frac{\alpha}{2}}, \tag{R$11$}\label{E:antip-R} \\
  (u^{-1})_\alpha^\alpha = (u_\alpha^{-\alpha})^{-1} = \sum_i (R'_i)_\alpha^{\frac{\alpha}{2}} S_{-\alpha}^{-\frac{\alpha}{2}}(S_\alpha^{\frac{\alpha}{2}}((R''_i)_\alpha^{\frac{\alpha}{2}})), \tag{R$12$}\label{E:inv-Drinf} \\
  u_\alpha^{-\alpha} x (u^{-1})_\alpha^\alpha = S_{-\alpha}^{-\beta} \left( S_\alpha^\beta (x) \right), \tag{R$13$}\label{E:conj-Drinf} \\
  \Delta_{\alpha,\beta}^0 \left( g_{\alpha+\beta} \right) 
  = g_\alpha \otimes g_\beta, \tag{R$14$}\label{E:coprod-piv} \\
  \varepsilon^0 \left( g_0 \right) 
  = 1, \tag{R$15$}\label{E:coun-piv} \\
  g_\alpha x g_\alpha^{-1} = S_{-\alpha}^{-\beta} \left( S_\alpha^\beta (x) \right) \tag{R$16$}\label{E:conj-piv}
 \end{gather}
 for every $x \in H_\alpha^\beta$.
\end{proposition}

\begin{proof}
  Equations~\eqref{E:YB}--\eqref{E:antip-R} are the Hopf $G$-bialgebra version of \cite[Theorem~VIII.2.4]{K95}, and the reader can directly adapt the proof from there. Notice that the Hopf $G$-coalgebra version of these properties is \cite[Proposition~1.4]{O93}. Equations~\eqref{E:inv-Drinf} and \eqref{E:conj-Drinf} are standard properties of the Drinfeld element, and can be adapted from \cite[Proposition~VIII.4.1]{K95}. Notice that the Hopf $G$-coalgebra version of these properties appears in \cite[Lemma 6.5]{V00}. Equations~\eqref{E:coprod-piv}--\eqref{E:conj-piv} are standard properties of the pivotal element, and can be adapted from \cite[Propositions~VIII.4.5 \& VIII.4.1]{K95} using the defining axioms of the ribbon element. Notice that the Hopf $G$-coalgebra version of these properties appears in \cite[Definition~2.1]{GHP22}.
\end{proof}

\subsection{Factorizability}\label{S:factorizability}

A ribbon Hopf $G$-bialgebra $H = \{ H_\alpha^\beta \mid \alpha,\beta \in G \}$ is \textit{factorizable} if it can be equipped with
\begin{itemize}
 \item a \textit{left integral} $\lambda = \{ \lambda_\alpha^0 : H_\alpha^0 \to \Bbbk \mid \alpha \in G \}$,
 \item a \textit{two-sided cointegral} $\Lambda = \{ \Lambda_0^\alpha : \Bbbk \to H_0^\alpha \mid \alpha \in G \}$.
\end{itemize}
For all $\alpha,\beta \in G$, these data are required to satisfy
\begin{gather*}
 \lambda_\beta^0(x_{(2,\beta)}) x_{(1,\alpha)} = \lambda_{\alpha+\beta}^0(x) 1_\alpha, \tag{I$1$}\label{E:int} \\
 y \Lambda_0^\beta = \varepsilon^\alpha(y) \Lambda_0^{\alpha+\beta}, \tag{I$2$}\label{E:coint} \\
 S_0^\alpha(\Lambda_0^\alpha) = \Lambda_0^{-\alpha}, \tag{I$3$}\label{E:antip-coint} \\
 \lambda_0^0(\Lambda_0^0) = 1, \tag{I$4$}\label{E:int-coint} \\
 D_{\alpha,0}^{\phantom{0}} (\lambda_{-\alpha}^0 \circ S_\alpha^0) = \Lambda_0^\alpha \tag{I$5$}\label{E:factorizable}
\end{gather*}
for all $x \in H_{\alpha+\beta}^0$ and $y \in H_0^\beta$, where
\begin{align*}
 D_{\alpha,\beta} : (H_\alpha^\beta)^* &\to H_\beta^\alpha \\*
 f &\mapsto f((R''_i)_\alpha^{\frac{\beta}{2}} (R'_j)_\alpha^{\frac{\beta}{2}})(R'_i)_\beta^{\frac{\alpha}{2}} (R''_j)_\beta^{\frac{\alpha}{2}}
\end{align*}
is the \textit{Drinfeld map}. We will use the shorthand notation 
\[
 \lambda_\alpha = \lambda_\alpha^0, \qquad \Lambda^\alpha = \Lambda_0^\alpha(1)
\]
for every $\alpha \in G$.

\begin{remark}\label{R:factorizable}
 If $H = \{ H_\alpha^\beta \mid \alpha,\beta \in G \}$ is a factorizable ribbon Hopf $G$-bialgebra, then:
 \begin{enumerate}
  \item $H_0^0$ is a factorizable ribbon Hopf algebra in the sense of \cite[Section~7]{H96}, thanks to \cite[Theorem~5]{K96};
  \item Every left integral $\lambda^\oplus = \{ \lambda_\alpha^\oplus : H_\alpha^\oplus \to \Bbbk \mid \alpha \in G \}$ for the Hopf $G$-coalgebra $H_\alpha^\oplus$, in the sense of \cite[Section~3.1]{V00}, determines a left integral $\lambda = \{ \lambda_\alpha : H_\alpha^0 \to \Bbbk \mid \alpha \in G \}$ for $H$, where $\lambda_\alpha$ is the restriction of $\lambda_\alpha^\oplus$ to $H_\alpha^0 \subset H_\alpha^\oplus$ for every $\alpha \in G$.
 \end{enumerate}
\end{remark}

\begin{proposition}\label{P:integrals}
 If $H = \{ H_\alpha^\beta \mid \alpha,\beta \in G \}$ is a factorizable ribbon Hopf $G$-bialgebra, then, for all $\alpha, \beta \in G$, the left integral satisfies
 \begin{gather}
  \lambda_\alpha(x_{(1,\alpha)}) x_{(2,\beta)} = \lambda_{\alpha+\beta}(x) g_\beta^2, \tag{I$6$}\label{E:unibalanced} \\
  \lambda_\alpha(yz) = \lambda_\alpha \left( z S_{-\alpha}^{-\beta} \left( S_\alpha^\beta(y) \right) \right), \tag{I$7$}\label{E:quantum-char} \\
  \lambda_\alpha ( w g_\alpha^{-1} ) = \lambda_{-\alpha} ( S_\alpha^0(w) g_{-\alpha}^{-1} ) \tag{I$8$}\label{E:antip-int}
 \end{gather}
 for all $x \in H_{\alpha+\beta}^0$, $y \in H_\alpha^\beta$, $x \in H_\alpha^{-\beta}$, $w \in H_\alpha^0$.
\end{proposition}

\begin{proof}
 The first claim is the Hopf $G$-bialgebra version of \cite[Theorem~10.2.2.(b)]{R12}. Similarly, Equations~\eqref{E:unibalanced} and \eqref{E:quantum-char} are the Hopf $G$-bialgebra version of \cite[Theorem~10.5.4.(b) \& (e)]{R12}, at least up to the Hopf $G$-bialgebra version of \cite[Exercise~10.2.1]{R12}, see also \cite[Theorem~7.18.12]{EGNO15}. Remark that the Hopf $G$-coalgebra version of these properties is a direct consequence\footnote{Notice that our pivotal element $g$ coincides with Virelizier's spherical grouplike element $G$, and not with its square $g$, called the distinguished grouplike element. Also Radford denotes by $g$ the square of our pivotal element.} of \cite[Lemma~4.1 \& Theorem 4.2]{V00}. To deduce Equations~\eqref{E:unibalanced}--\eqref{E:antip-int} from there, the Hopf $G$-bialgebra version of \cite[Theorem~1]{KR93} is also needed.
\end{proof}

\section{Construction of the TQFT functor}\label{S:Construction_of_the_functor}

Let $H  = \{ H_\alpha^\beta \mid \alpha,\beta \in G \}$ be a factorizable ribbon Hopf $G$-bialgebra over a field $\Bbbk$. We will now construct a braided monoidal TQFT functor
\[
 J_H : \Tan^G \to \mods{H_{\mathrm{0}}^\oplus}.
\]

On the level of objects, for every finite sequence
\[
 (\myuline{\alpha},\myuline{\beta}) = (\alpha_1,\beta_1,\ldots,\alpha_g,\beta_g) \in G^{\times 2g},
\]
let us set 
\begin{equation}
 J_H(\myuline{\alpha},\myuline{\beta}) := H_{(\myuline{\alpha},\myuline{\beta})} = \bigotimes_{i=1}^g \transm{H}_{\alpha_i}^{\beta_i}, \label{E:image_objects}
\end{equation}
where, for all $\alpha,\beta \in G$, the $H_0^\oplus$-module $\transm{H}_\alpha^\beta$ is given by the vector space $H_\alpha^\beta$ equipped with the adjoint action 
\[
 x \triangleright y = x_{(1,\alpha)}yS_{-\alpha}^\gamma(x_{(2,-\alpha)})
\]
for all $x \in H_0^\gamma$ and $y \in \transm{H}_\alpha^\beta$.

On the level of morphisms, the definition requires some more preparation. First of all, suppose that, for a fixed regular diagram of a top tangle, the number of strands of running along a $\beta$-labeled $1$-han\-dle $B$ is $k$, with orientations and labels 
\[ 
 (\myuline{\sigma},\myuline{\xi}) = ((\sigma_1,\xi_1),\ldots,(\sigma_k,\xi_k)) \in (\{ +,- \} \times G)^{\times k}.
\]
Our convention here is that positive strands are those that run along $B$ from right to left, and we set $\alpha = \sigma_1 \xi_1 + \ldots + \sigma_k \xi_k$. Let us agree that \textit{inserting a bead with label $x \in \transm{H}_\alpha^\beta$ on the $1$-handle $B$} means considering the $(k-1)$-iterated oriented coproduct
\begin{align*}
 \Delta_{(\myuline{\sigma},\myuline{\xi})}^\beta(x) 
 &:= x_{(1,\xi_1)}^{\sigma_1} \otimes \ldots \otimes x_{(k,\xi_k)}^{\sigma_k} \in \bigotimes_{i=1}^k H_{\xi_i}^{\sigma_i \beta}, \\
 x_{(i,\xi_i)}^{\sigma_i} 
 &:= 
 \begin{cases}
  x_{(i,\xi_i)} \in H_{\xi_i}^\beta & \mbox{ if } \sigma_i = +, \\
  S_{-\xi_i}^\beta (x_{(i,-\xi_i)}) \in H_{\xi_i}^{-\beta} & \mbox{ if } \sigma_i = -,
 \end{cases} \qquad \Forall 1 \leqs i \leqs k,
\end{align*}
and then decorating the diagram around $B$ as shown:
\[
 \pic{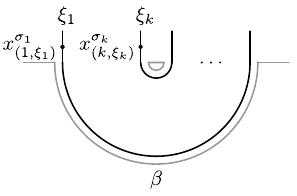}
\]
When $k = 0$, we add by convention a multiplicative factor of $\varepsilon^\beta(x)$ in front of the whole diagram. Here is an example of bead insertion:
\[
 \pic[-17.25]{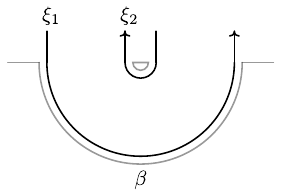} \mapsto \pic[-17.25]{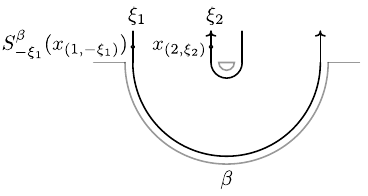}
\]
In this example, we are inserting a bead with label $x \in \transm{H}_\alpha^\beta$ for $\alpha = -\xi_1+\xi_2$. Using this operation, let us define, for a fixed regular diagram of a top tangle
\[
 T \in \Tan^G((\myuline{\alpha},\myuline{\beta}),(\myuline{\alpha'},\myuline{\beta'})),
\]
a linear map\footnote{We will need to show later that $J_H(T)$ is indeed well-defined, meaning that it does not depend on the chosen diagram, that it is invariant invariant under $G$-Kirby moves, and that it is an $H_0^\oplus$-intertwiner.}
\[
 J_H(T) : H_{(\myuline{\alpha},\myuline{\beta})} \to H_{(\myuline{\alpha'},\myuline{\beta'})}.
\]
In order to do this, let us consider a vector
\[
 x_{1} \otimes \ldots \otimes x_g \in H_{(\myuline{\alpha},\myuline{\beta})}.
\]
First, we insert a bead with label $x_i \in \transm{H}_{\alpha_i}^{\beta_i}$ on the $1$-handle $B_i$ of $M_g$ for every integer $1 \leqs i \leqs g$. Next, we insert beads labeled by components of the R-matrix around crossings, as shown:
\begin{align*}
 \pic{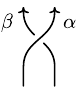} &\mapsto \pic{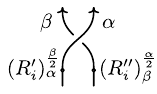} &
 \pic{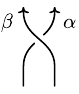} &\mapsto \pic{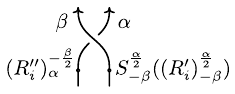} \\
 \pic{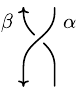} &\mapsto \pic{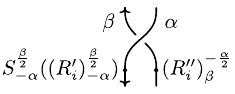} &
 \pic{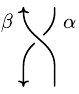} &\mapsto \pic{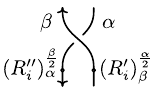} \\
 \pic{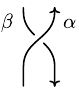} &\mapsto \pic{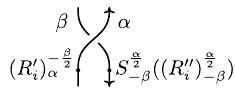} &
 \pic{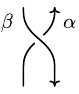} &\mapsto \pic{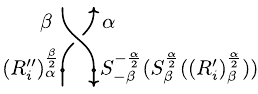} \\
 \pic{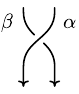} &\mapsto \pic{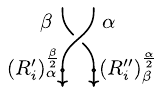} &
 \pic{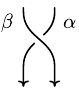} &\mapsto \pic{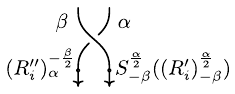}
\end{align*}
Then, we insert beads labeled by pivotal elements around right-oriented extrema, as shown:
\begin{align*}
 \pic{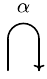} &\mapsto \pic{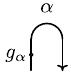} &
 \pic{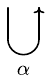} &\mapsto \pic{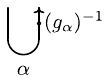}
\end{align*}
Notice that this step has to be applied also to all the right-oriented minima that should appear inside $1$-handles. Next, we collect all beads sitting on the same component in one place, by sliding them along the strand without changing their order, and we multiply everything together according to the rule
\begin{align*}
 \pic{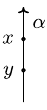} &= \pic{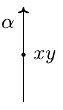}
\end{align*}
In the end, we are left with a top tangle $B(T)$ carrying at most a single bead on each of its components. The linear map $J_H(T) : H_{(\myuline{\alpha},\myuline{\beta})} \to H_{(\myuline{\alpha'},\myuline{\beta'})}$ is determined by
\begin{equation}
 J_H(T)(x_1 \otimes \ldots \otimes x_g) = 
 \left( \prod_{k=1}^h \lambda_{\gamma_k}(y_k g_{\gamma_k}^{-1}) \right)
 x'_1 \otimes \ldots \otimes x'_{g'} \in H_{(\myuline{\alpha'},\myuline{\beta'})} \label{E:image_morphisms}
\end{equation}
for every $x_1 \otimes \ldots \otimes x_g \in H_{(\myuline{\alpha},\myuline{\beta})}$, where
\begin{equation*}
 B(T) = \pic{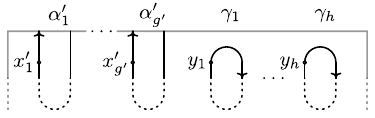}
\end{equation*}

\begin{lemma}
 In Equation~\eqref{E:image_morphisms}, we have indeed
 \[
  x'_j \in \transm{H}_{\alpha'_j}^{\beta_j'}, \qquad
  y_k \in \transm{H}_{\gamma_k}^0
 \]
 for all $1 \leqs j \leqs g'$ and $1 \leqs k \leqs h$.
\end{lemma}

\begin{proof}
 The arc component $A'_\ell$ of $T$ runs along the $1$-handle $B_i$ of $M_g$ an algebraic total of $\lk(a_i,\hat{A}'_\ell)$ times, thus picking up beads whose lower degrees are all $\alpha'_\ell$, and whose total upper degree is $\lk(a_i,\hat{A}'_\ell) \beta_i$. Afterwards, it crosses the arc component $A'_j$ and the closed component $C_k$ of $T$ an algebraic total of $2 \lk(\hat{A}'_j,\hat{A}'_\ell)$ and of $2 \lk(C_k,\hat{A}'_\ell)$ times, respectively, thus picking up beads whose lower degree is always $\alpha'_\ell$, and whose total upper degree is $\lk(\hat{A}'_j,\hat{A}'_\ell) \alpha'_j$ and $\lk(C_k,\hat{A}'_\ell) \gamma_k$, respectively. Then, it picks up pivotal elements with lower degree $\alpha'_\ell$ and upper degree $0$. Since products preserve lower degrees and add up upper degrees, we deduce that $x'_\ell \in \transm{H}_{\alpha'_\ell}^{\beta_\ell'}$, as follows from Equation~\eqref{E:target}. The same argument applied to the circle component $C_\ell$ of $T$ shows that $y_\ell \in \transm{H}_{\gamma_k}^0$, as follows from Equation~\eqref{E:middle}.
\end{proof}

\begin{theorem}\label{T:main}
 Equations~\eqref{E:image_objects} and \eqref{E:image_morphisms} define a braided monoidal functor 
 \[
  J_H : \Tan^G \to \mods{H_{\mathrm{0}}^\oplus}.
 \]
\end{theorem}

\subsection{Proof of the invariance}\label{S:invariance}

Let us prove our main result.

\begin{proof}[Proof of Theorem~\ref{T:main}]
 We need to show that:
 \begin{itemize}
  \item $J_H(T)$ is independent of the regular diagram representing $T$;
  \item $J_H(T)$ is invariant under framed $G$-Kirby moves on $T$;
  \item $J_H(T' \circ T) = J_H(T') \circ J_H(T)$ and $J_H(\id_{(\myuline{\alpha},\myuline{\beta})}) = \id_{J_H(\myuline{\alpha},\myuline{\beta})}$;
  \item $J_H(T)$ is an $H_0^\oplus$-intertwiner;
  \item $J_H(T \bcs T) = J_H(T) \otimes J_H(T')$ and $J_H(\varnothing) = \Bbbk$;
  \item $J_H(c_{(\myuline{\alpha},\myuline{\beta}),(\myuline{\alpha'},\myuline{\beta'})}) = c_{J_H(\myuline{\alpha},\myuline{\beta}),J_H(\myuline{\alpha'},\myuline{\beta'})}$.
 \end{itemize}
 Invariance under framed Reidemeister moves follows from the Hopf $G$-bialgebra version of the standard argument that shows the isotopy invariance of Lawrence's universal invariant of links, using Equations~\eqref{E:square-rib}, \eqref{E:inv-R}, and \eqref{E:YB}, compare with the proof of \cite[Theorem~4.1]{O93}. Here, we also have to check that $1$-handles are dinatural with respect to crossings and extrema. For what concerns crossings, we need to compare the following tangles:
 \[
  \pic{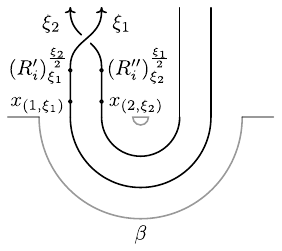} \qquad \pic{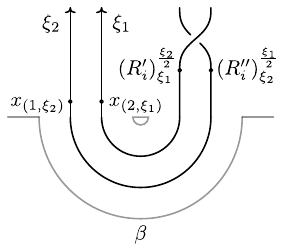}
 \]
 The two configurations lead to the same result, as a consequence of Equations~\eqref{E:R-intertw-coprod} and \eqref{E:inv-R}, and the identity for the other crossings is proved similarly. For what concerns maxima, we need to compare the following tangles:
 \[
  \pic{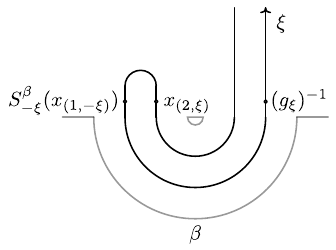} \qquad \pic{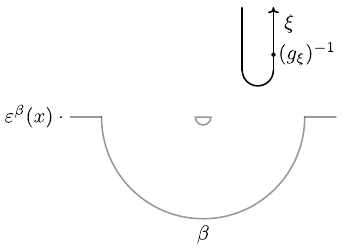}
 \]
 Again, the two configurations lead to identical results, as a consequence of Equation~\eqref{E:antip}, and the identity for the other orientation is proved similarly. The same goes for maxima, which lead us to compare the following tangles:
 \[
  \pic{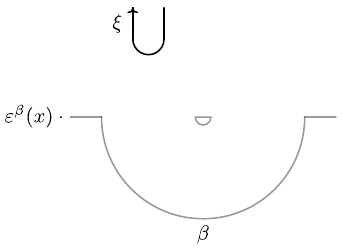} \qquad \pic{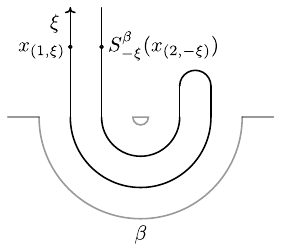}
 \]

 Next, we need to prove invariance under framed $G$-Kirby moves. The proof is essentially the same as the one for \cite[Theorem~3.1]{BD22}. Indeed, in order to prove invariance under \eqref{E:K1}, let us start by comparing the following tangles:
 \[
  \pic{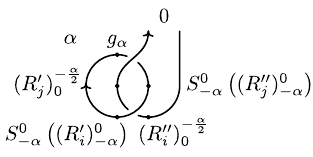} \qquad \pic{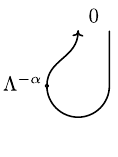}
 \]
 The two configurations lead to the same result, as a consequence of Equation~\eqref{E:factorizable}, since
 \begin{align*}
  &\lambda_\alpha \left( S_{-\alpha}^0 \left( (R'_i)_{-\alpha}^0 \right) S_{-\alpha}^0 \left( (R''_j)_{-\alpha}^0 \right) \right) (R'_j)_0^{-\frac{\alpha}{2}} (R''_i)_0^{-\frac{\alpha}{2}} \\
  &\hspace*{\parindent} = \lambda_\alpha \left( S_{-\alpha}^0 \left( (R''_j)_{-\alpha}^0 (R'_i)_{-\alpha}^0 \right) \right) (R'_j)_0^{-\frac{\alpha}{2}} (R''_i)_0^{-\frac{\alpha}{2}} \\
  &\hspace*{\parindent} = \Lambda^{-\alpha}.
 \end{align*}
 This means that
 \[
  \pic{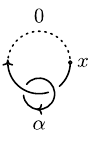} \rightsquigarrow \pic{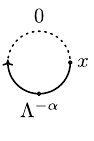}
 \]
 Then the invariance follows from Equations~\eqref{E:coint}, \eqref{E:coun-R}, and \eqref{E:coun-piv}. For what concerns \eqref{E:K2}, we have
 \[
  \pic{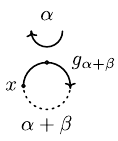} \rightsquigarrow \pic{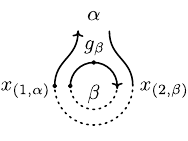}
 \]
 Then the invariance follows from Equation~\eqref{E:int}. Finally, for what concerns \eqref{E:K3}, we use the \textit{handle trick}, which amounts to replace the left-hand tangle with the right-hand one, here below:
 \[
  \pic{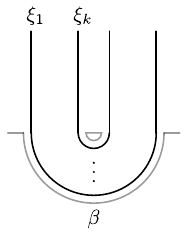} \qquad \pic{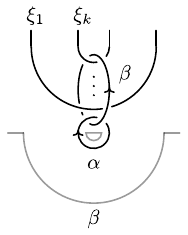}
 \]
 As usual, we are assuming that $\alpha = \sigma_1 \xi_1 + \ldots + \sigma_k \xi_k$. Notice that the two configurations lead to the same result, as a consequence of the invariance under \eqref{E:K1} and \eqref{E:K2}. Therefore, we can always assume that the circle component whose orientation we want to reverse is entirely contained in the unit cube $[0,1]^{\times 3} \subset \R^3$. Then, up to isotopy, we can assume that the knot is written as the closure of an upward-oriented braid (recall that we already explained that the result of the algorithm is invariant under isotopies). For a braid closure, reversing the orientation amounts to replacing each bead labeled by $x \in H_\alpha^\beta$ with a bead labeled by $S_\alpha^\beta(x) \in H_{-\alpha}^{-\beta}$. This means that
 \[
  \pic{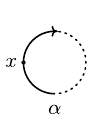} \rightsquigarrow \pic{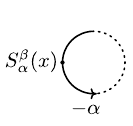}
 \]
 Then the invariance follows from Equation~\eqref{E:antip-int}. 

 Next, we need to prove that $J_H$ is a functor. The fact that it respects compositions can be proved using the handle trick once again, while the fact that it respects identities is clear from the algorithm.

 Next, we need to prove that $J_H(T)$ is an $H_0^\oplus$-intertwiner. Thanks to Remark~\ref{R:generators}, it is enough to check this for generators of $\Tan^G$. A computation shows that those appearing in Figure~\ref{F:generators} are sent to:
 \begin{itemize}
  \item a product 
   \begin{align*}
    \transm{\mu}_\alpha^{\beta,\gamma} : \transm{H}_\alpha^\beta \otimes \transm{H}_\alpha^\gamma &\to \transm{H}_\alpha^{\beta+\gamma} \\*
    x \otimes y &\mapsto xy;
   \end{align*}
  \item a unit
   \begin{align*}
    \transm{\eta}_\alpha : \Bbbk &\to \transm{H}_\alpha^0 \\*
    1 &\mapsto 1_\alpha;
   \end{align*}
  \item a coproduct 
   \begin{align*}
    \transm{\Delta}_{\alpha,\beta}^\gamma : \transm{H}_{\alpha+\beta}^\gamma &\to \transm{H}_\alpha^\gamma \otimes \transm{H}_\beta^\gamma \\*
    x &\mapsto \sum_i x_{(1,\alpha)}S_{-\alpha}^0((R''_i)_{-\alpha}^0) \otimes ((R'_i)_0^{-\frac{\alpha}{2}} \triangleright x_{(2,\beta)});
   \end{align*}
  \item a counit
   \begin{align*}
    \transm{\varepsilon}^\alpha : \transm{H}_0^\alpha &\to \Bbbk \\*
    x &\mapsto \varepsilon^\alpha(x);
   \end{align*}
  \item an antipode
   \begin{align*}
    \transm{S}_\alpha^\beta : \transm{H}_\alpha^\beta &\to \transm{H}_{-\alpha}^{-\beta} \\*
    x &\mapsto \sum_i (R''_i)_{-\alpha}^0 S_\alpha^\beta((R'_i)_0^{-\frac{\alpha}{2}} \triangleright x);
   \end{align*}
  \item an inverse antipode
   \begin{align*}
    (\myuline{S^{-1}})_\alpha^\beta : \transm{H}_\alpha^\beta &\to \transm{H}_{-\alpha}^{-\beta} \\*
    x &\mapsto \sum_i (S^{-1})_\alpha^\beta((R'_i)_0^{-\frac{\alpha}{2}} \triangleright x)(R''_i)_{-\alpha}^0;
   \end{align*}
  \item a ribbon element
   \begin{align*}
    \transm{v}_\alpha^{-\alpha} : \Bbbk &\to \transm{H}_\alpha^{-\alpha} \\*
    1 &\mapsto v_\alpha^{-\alpha};
   \end{align*}
  \item an inverse ribbon element
   \begin{align*}
    (\myuline{v^{-1}})_\alpha^\alpha : \Bbbk &\to \transm{H}_\alpha^\alpha \\*
    1 &\mapsto (v_\alpha^{-\alpha})^{-1};
   \end{align*}
  \item a counit
   \begin{align*}
    \transm{\lambda}_\alpha : \transm{H}_\alpha^0 &\to \Bbbk \\*
    x &\mapsto \lambda_\alpha(x).
   \end{align*}
 \end{itemize}
 For ribbon Hopf algebras, which correspond to the case $G = 0$, it was first proved in \cite{M91} that (the analogue of) these linear maps intertwine tensor powers of the adjoint action, see also \cite[Example~9.4.9]{M95}. Let us check this for the coproduct and the antipode, and leave the remaining generators to the reader. For all $x \in H_0^\delta$ and $y \in \transm{H}_{\alpha+\beta}^\gamma$, we have
 \begin{align*}
  &\transm{\Delta}_{\alpha,\beta}^\gamma(x \triangleright y) \\
  &= \transm{\Delta}_{\alpha,\beta}^\gamma
  (x_{(1,\alpha+\beta)} y S_{-\alpha-\beta}^\delta(x_{(2,-\alpha-\beta)})) \\
  &= \sum_i x_{(1,\alpha)} y_{(1,\alpha)} S_{-\alpha}^\delta(x_{(4,-\alpha)}) S_{-\alpha}^0((R''_i)_{-\alpha}^0) 
  \otimes ((R'_i)_0^{-\frac{\alpha}{2}} \triangleright (x_{(2,\beta)} y_{(2,\beta)} S_{-\beta}^\delta(x_{(3,-\beta)})) \\
  &= \sum_i x_{(1,\alpha)} y_{(1,\alpha)} S_{-\alpha}^\delta((R''_i)_{-\alpha}^0 x_{(3,-\alpha)}) 
  \otimes (((R'_i)_0^{-\frac{\alpha}{2}} x_{(2,0)}) \triangleright y_{(2,\beta)}) \\
  &= \sum_i x_{(1,\alpha)} y_{(1,\alpha)} S_{-\alpha}^\delta(x_{(2,-\alpha)} (R''_i)_{-\alpha}^0) 
  \otimes ((x_{(3,0)} (R'_i)_0^{-\frac{\alpha}{2}} )\triangleright y_{(2,\beta)}) \\
  &= \sum_i x_{(1,\alpha)} y_{(1,\alpha)} S_{-\alpha}^0((R''_i)_{-\alpha}^0) S_{-\alpha}^\delta(x_{(2,-\alpha)}) 
  \otimes (x_{(3,0)} \triangleright ((R'_i)_0^{-\frac{\alpha}{2}} \triangleright y_{(2,\beta)})) \\
  &= \sum_i (x_{(1,0)} \triangleright (y_{(1,\alpha)} S_{-\alpha}^0((R''_i)_{-\alpha}^0))) 
  \otimes (x_{(2,0)} \triangleright ((R'_i)_0^{-\frac{\alpha}{2}} \triangleright y_{(2,\beta)})) \\
  &= x \triangleright \transm{\Delta}_{\alpha,\beta}^\gamma(y),
 \end{align*}
 where the second equality follows from Equation~\eqref{E:antip-coprod}, the third and fifth ones from Equation~\eqref{E:antip-prod}, and the fourth one from Equation~\eqref{E:R-intertw-coprod}. Similarly, for all $x \in H_0^\gamma$ and $y \in \transm{H}_\alpha^\beta$, we have
 \begin{align*}
  &\transm{S}_\alpha^\beta (x \triangleright y) \\
  &= \transm{S}_\alpha^\beta
  (x_{(1,\alpha)} y S_{-\alpha}^\gamma(x_{(2,-\alpha)})) \\
  &= \sum_i (R''_i)_{-\alpha}^0 S_\alpha^\beta((R'_i)_0^{-\frac{\alpha}{2}} \triangleright (x_{(1,\alpha)} y S_{-\alpha}^\gamma(x_{(2,-\alpha)}))) \\
  &= \sum_{j,k} (R''_k)_{-\alpha}^{\frac{\alpha}{2}} (R''_j)_{-\alpha}^{-\frac{\alpha}{2}} S_\alpha^\beta((R'_k)_\alpha^{-\frac{\alpha}{2}} x_{(1,\alpha)} y S_{-\alpha}^\gamma(x_{(2,-\alpha)}) S_{-\alpha}^{-\frac{\alpha}{2}}((R'_j)_{-\alpha}^{-\frac{\alpha}{2}})) \\
  &= \sum_k (R''_k)_{-\alpha}^{\frac{\alpha}{2}} (u_{-\alpha}^\alpha)^{-1} S_\alpha^{-\gamma}(S_{-\alpha}^\gamma(x_{(2,-\alpha)})) S_\alpha^{-\frac{\alpha}{2}+\beta+\gamma}((R'_k)_\alpha^{-\frac{\alpha}{2}} x_{(1,\alpha)} y) \\
  &= \sum_k (R''_k)_{-\alpha}^{\frac{\alpha}{2}} x_{(2,-\alpha)} (u_{-\alpha}^\alpha)^{-1} S_\alpha^{-\frac{\alpha}{2}+\beta+\gamma}((R'_k)_\alpha^{-\frac{\alpha}{2}} x_{(1,\alpha)} y) \\
  &= \sum_k x_{(1,-\alpha)} (R''_k)_{-\alpha}^{\frac{\alpha}{2}} (u^{-1})_{-\alpha}^{-\alpha} S_\alpha^{-\frac{\alpha}{2}+\beta+\gamma}(x_{(2,\alpha)} (R'_k)_\alpha^{-\frac{\alpha}{2}} y) \\
  &= x \triangleright \transm{S}_\alpha^\beta (y),
 \end{align*}
where the third equality follows from Equation~\eqref{E:R-coprod-left}, the fourth and sixth ones from Equations~\eqref{E:antip-prod} and \eqref{E:inv-Drinf}, and the fifth one from Equation~\eqref{E:conj-Drinf}.
 
 Finally, we need to prove that $J_H$ is braided monoidal. Notice however that the fact that $J_H$ respects tensor products, tensor unit, and braiding is clear from the algorithm.
\end{proof}

\subsection{Relation with the Kerler--Lyubashenko TQFT}\label{S:relation_with_Lyubashenko}

If $H  = \{ H_\alpha^\beta \mid \alpha,\beta \in G \}$ is a factorizable ribbon Hopf $G$-bialgebra over a field $\Bbbk$, then, thanks to Remark~\ref{R:factorizable}, $H_0^0$ is a ribbon Hopf algebra. In particular, $H_0^0$ induces a Kerler--Lyubashenko TQFT
\[
 J_{H_0^0} : \Tan \to \mods{H_{\mathrm{0}}^{\mathrm{0}}}
\]
thanks to the construction of \cite{KL01}, see also \cite[Sections~4.2.3 \& 4.2.4]{DM22}. Notice that the inclusion $0 \hookrightarrow G$ induces an inclusion functor
\[
 \iota : \Tan \to \Tan^G,
\]
and similarly the inclusion $H_0^0 \hookrightarrow H_0^\oplus$ induces a restriction functor
\[
 \rho : \mods{H_{\mathrm{0}}^\oplus} \to \mods{H_{\mathrm{0}}^{\mathrm{0}}}.
\]

\begin{corollary}\label{C:relation_with_KL}
 The diagram of braided ribbon functors
 \begin{center}
  \begin{tikzpicture}[descr/.style={fill=white}] \everymath{\displaystyle}
   \node (P0) at (0,0) {$3\Cob$};
   \node (P1) at (2,0) {$\mods{H_{\mathrm{0}}^{\mathrm{0}}}$};
   \node (P2) at (0,-1) {$3\Cob^G$};
   \node (P3) at (2,-1) {$\mods{H_{\mathrm{0}}^\oplus}$};
   \draw
   (P0) edge[->] node[above] {\scriptsize $J_{H_0^0}$}(P1)
   (P0) edge[->] node[left] {\scriptsize $\iota$} (P2)
   (P3) edge[->] node[right] {\scriptsize $\rho$} (P1)
   (P2) edge[->] node[below] {\scriptsize $J_H$} (P3);
  \end{tikzpicture}
 \end{center}
 is commutative.
\end{corollary}

\begin{proof}
 The proof is almost clear, except for the fact that the algorithm presented in \cite[Section~4.2.4]{DM22} is given in terms of unoriented top tangles in handlebodies, instead of oriented ($0$-labeled) top tangles in ($0$-labeled) handlebodies. The two recipes do not coincide, but it turns out that they yield the same result. Indeed, this is true when the two algorithms are applied to the oriented generators obtained from Figure~\ref{F:generators} by setting all labels equal to $0$, on one hand, and to their unoriented unlabeled version, on the other hand. Then, since both procedures produce braided monoidal functors, their result must coincide for arbitrary top tangles too.
\end{proof}

\section{Quantum \texorpdfstring{$\fsl_2$}{sl(2)}}\label{S:quantum_sl2_algebra}

Let us consider the primitive $r$th root of unity $q = e^{\frac{2 \pi \fraki}{r}}$, where $r \geqs 3$ is an integer, and let us set $r' := r/\gcd(2,r)$. For every $\alpha \in \C$, and for all positive integers $k \geqs \ell \geqs 0$, we recall the notation 
\begin{align*}
 q^\alpha &:= e^{\frac{2 \alpha \pi \fraki}{r}}, &
 \{ \alpha \} &:= q^\alpha - q^{-\alpha}, &
 [k] &:= \frac{\{ k \}}{\{ 1 \}}, &
 [k]! &:= \prod_{j=1}^k [j], &
 \sqbinom{k}{\ell} &:= \frac{[k]!}{[\ell]![k-\ell]!}.
\end{align*}
The \textit{unrestricted quantum group} $\frakU = \frakU_q \fsl_2$ is the $\C$-algebra with generators
\[
 \{ E,F,K^\alpha \mid \alpha \in \C \}
\]
and relations
\begin{gather*}
 E^{r'} = F^{r'} = 0, \qquad K^\alpha K^\beta = K^{\alpha+\beta}, K^0 = 1, \\*
 K^\alpha E K^{-\alpha} = q^{2\alpha} E, \qquad K^\alpha F K^{-\alpha} = q^{-2\alpha} F, \qquad
 [E,F] = \frac{K - K^{-1}}{q - q^{-1}}.
\end{gather*}
Now $\frakU$ admits a Hopf algebra structure obtained by setting
\begin{align*}
 \Delta(E) &= E \otimes K + 1 \otimes E, & \varepsilon(E) &= 0, & S(E) &= -E K^{-1}, \\*
 \Delta(F) &= K^{-1} \otimes F + F \otimes 1, & \varepsilon(F) &= 0, & S(F) &= - K F, \\*
 \Delta(K^\alpha) &= K^\alpha \otimes K^\alpha, & \varepsilon(K^\alpha) &= 1, & S(K^\alpha) &= K^{-\alpha}.
\end{align*}
Let us set $G = \C/2\Z$ and $\frakU_{\bar{\alpha}} := \frakU/(K^{\frac{r}{2}}-q^{\frac{\alpha r}{2}})$ for every $\bar{\alpha} \in G$, where $\alpha \in \C$ is any representative of $\bar{\alpha}$. Next, let us consider the subalgebra $\fraku_{\bar{\alpha}}$ of $\frakU_{\bar{\alpha}}$ generated by $\{ E,F,K \}$, and set $\fraku_{\bar{\alpha}}^{\bar{\beta}} := \fraku_{\bar{\alpha}} K^\beta$ for every $\bar{\beta} \in G$, where $\beta \in \C$ is any representative of $\bar{\beta}$. Then $\fraku := \{ \fraku_{\bar{\alpha}}^{\bar{\beta}} \mid \bar{\alpha},\bar{\beta} \in G \}$ inherits the structure of a Hopf $G$-bialgebra. A Poincaré--Birkhoff--Witt basis for $\fraku_{\bar{\alpha}}^{\bar{\beta}}$ is given by
\[
 \{ E^\ell F^m K^{n+\beta} \mid 0 \leqs \ell,m,n \leqs r'-1 \}.
\]

\begin{remark}
 When $r \equiv 1 \pmod 2$, a Poincaré--Birkhoff--Witt basis for $\fraku_{\bar{\alpha}}^0 = \fraku_{\bar{\alpha}}$ is given by
 \[
  \{ E^\ell F^m K^{\frac{n}{2}} \mid 0 \leqs \ell,m,n \leqs r-1 \},
 \]
 because $K^n = q^{\frac{\alpha r}{2}} K^{n-\frac{r}{2}}$ for every $\frac{r-1}{2} < n \leqs r-1$. In this case,
 \begin{align*}
  \fraku_{\bar{\alpha}+1} &\to \fraku_{\bar{\alpha}} \\*
  E &\mapsto E \\*
  F &\mapsto F \\*
  K^{\frac{1}{2}} &\mapsto -K^{\frac{1}{2}}
 \end{align*}
 is an isomorphism of $\C$-algebras, and $\fraku_{\bar{\alpha}}^{\bar{\beta}+1} = \fraku_{\bar{\alpha}}^{\bar{\beta}}$.
\end{remark}

\subsection{Integral bases}\label{S:half-divided_basis}

For all $\alpha, \beta \in \C$, let us set
\begin{align*}
 T_\alpha^\beta := \frac{1}{r'} \sum_{b=0}^{r'-1} q^{-\alpha b} K^{b+\beta}.
\end{align*}

\begin{lemma}\label{L:projectors}
 In $\fraku_{\bar{\alpha}}^{\bar{\beta}}$ we have
 \begin{gather*}
  T_\alpha^\beta E = q^{2\beta} E T_{\alpha-2}^\beta, \qquad
  T_\alpha^\beta F = q^{-2\beta} F T_{\alpha+2}^\beta, \\*
  T_\alpha^\beta K = K T_\alpha^\beta = T_\alpha^{\beta+1} = q^\alpha T_\alpha^\beta, \qquad
  T_\alpha^\beta T_{\alpha'}^{\beta'} = \delta_{\alpha,\alpha'} T_\alpha^{\beta+\beta'} 
 \end{gather*}
\end{lemma}

\begin{proof}
 \begin{align*}
  T_\alpha^\beta E 
  &= \frac{1}{r'} \sum_{b=0}^{r'-1} q^{-\alpha b} K^{b+\beta} E
  = \frac{q^{2\beta}}{r'} \sum_{b=0}^{r'-1} q^{-(\alpha-2)b} E K^{b+\beta}
  = q^{2\beta} E T_{\alpha-2}^\beta, \\
  T_\alpha^\beta F 
  &= \frac{1}{r'} \sum_{b=0}^{r'-1} q^{-\alpha b} K^{b+\beta} F 
  = \frac{q^{-2\beta}}{r'} \sum_{b=0}^{r'-1} q^{-(\alpha+2)b} F K^{b+\beta}
  = q^{-2\beta} F T_{\alpha+2}^\beta, \\
  T_\alpha^\beta K 
  &= \frac{1}{r'} \sum_{b=0}^{r'-1} q^{-\alpha b} K^{b+1+\beta}  
  = K T_\alpha^\beta
  = T_\alpha^{\beta+1} \\
  &= \frac{q^\alpha}{r'} \left( q^{-\alpha r'} q^{\alpha r'} K^\beta + \sum_{b=1}^{r'-1} q^{-\alpha b} K^{b+\beta} \right) \\
  &=q^\alpha T_\alpha^\beta,\\
  T_\alpha^\beta T_{\alpha'}^{\beta'} 
  &= \frac{1}{r'} \sum_{b=0}^{r'-1} q^{-\alpha b} K^{b+\beta} T_{\alpha'}^{\beta'}
  = \frac{1}{r'} \sum_{b=0}^{r'-1} q^{-\alpha b} K^b T_{\alpha'}^{\beta+\beta'} 
  = \frac{1}{r'} \sum_{b=0}^{r'-1} q^{-\alpha b} q^{\alpha' b} T_{\alpha'}^{\beta+\beta'} \\
  &= \delta_{\alpha,\alpha'} T_\alpha^{\beta+\beta'} \qedhere
 \end{align*}
\end{proof}

Notice that, for every integer $0 \leqs b \leqs r-1$ and all complex numbers $\alpha, \beta \in \C$, we have
\begin{align}
 \sum_{a=0}^{r'-1} q^{(2a+\alpha)b} T_{2a+\alpha}^\beta
 &= \sum_{a,c=0}^{r'-1} q^{(2a+\alpha)b} q^{-(2a+\alpha)c} K^{c+\beta} \nonumber \\
 &= \sum_{c=0}^{r'-1} \left( \sum_{a=0}^{r'-1} q^{2a(b-c)} \right) q^{\alpha(b-c)} K^{c+\beta} \nonumber \\
 &= \sum_{c=0}^{r'-1} \delta_{b,c} q^{\alpha(b-c)} K^{c+\beta} \nonumber \\
 &= K^{b+\beta}. \label{E:from_T_to_K}
\end{align}
If, for every integer $0 \leqs a \leqs r-1$, we set
\[
 F^{(a)} := \frac{\{ 1 \}^a}{[a]!} F^a,
\]
then
\[
 \{ E^\ell F^{(m)} T_{2n+\alpha}^\beta \mid 0 \leqs \ell,m,n \leqs r'-1 \}
\]
is a basis of $\fraku_{\bar{\alpha}}^{\bar{\beta}}$ for any pair of representatives $\alpha,\beta \in \C$ of $\bar{\alpha},\bar{\beta} \in G$.

\subsection{Computations in integral bases}\label{S:Identities_in_hdb}

First of all, we compute coproducts and antipodes in integral bases.

\begin{lemma}\label{L:coproducts_antipodes}
 For all integers $0 \leqs a,b,c \leqs r-1$ we have
 \begin{align}
  \Delta \left( F^{(a)} E^b T_\alpha^\beta \right) 
  &= \sum_{c=0}^{r'-1} \sum_{i=0}^a \sum_{j=0}^b \sqbinom{b}{j} q^{(a-\alpha)i+bj-2c(i+j)-(i+j)^2} \nonumber \\*
  &\hspace*{\parindent} F^{(a-i)} E^j T_{2c+\alpha}^\beta \otimes F^{(i)} E^{b-j} T_{-2c}^\beta, \label{E:coproducts} \\*    
  S \left( F^{(a)} E^b T_\alpha^\beta \right) 
  &= (-1)^{a+b} q^{(a-b-\alpha-1)(a-b)} T_{-\alpha}^{-\beta} E^b F^{(a)}. \label{E:antipodes}
 \end{align}
\end{lemma}

\begin{proof}
 \cite[Chapter~3, Proposition~5]{KS97} gives
 \begin{align*}
  \Delta \left( F^\ell K^m E^n \right) 
  &= \sum_{i=0}^\ell \sum_{j=0}^n \sqbinom{\ell}{i} \sqbinom{n}{j} q^{i(\ell-i)-j(n-j)} F^{\ell-i} K^{m-i} E^j \otimes F^i K^{m+j} E^{n-j}, \\*
  S \left( F^\ell K^m E^n \right) 
  &= (-1)^{\ell+n} q^{\ell(\ell-1)-n(n-1)} E^n K^{\ell-m-n} F^\ell.
 \end{align*}
 Furthermore
 \begin{align*}
  \Delta \left( T_\alpha^\beta \right) 
  &= \frac{1}{r'} \sum_{b=0}^{r'-1} q^{-\alpha b} \Delta \left( K^{b+\beta} \right) 
  = \frac{1}{r'} \sum_{b=0}^{r'-1} q^{-\alpha b} K^{b+\beta} \otimes K^{b+\beta} \\*
  &= \frac{1}{r'} \sum_{a,b=0}^{r'-1} q^{2ab} T_{2a+\alpha}^\beta \otimes K^{b+\beta} 
  = \sum_{a=0}^{r'-1} T_{2a+\alpha}^\beta \otimes T_{-2a}^\beta, \\*
  S \left( T_\alpha^\beta \right) 
  &= \frac{1}{r'} \sum_{b=0}^{r'-1} q^{-\alpha b} S \left( K^{b+\beta} \right)
  = \frac{1}{r'} \sum_{b=0}^{r-1} q^{-\alpha b} K^{-b-\beta}
  = \frac{1}{r'} \sum_{b=0}^{r-1} q^{\alpha b} K^{b-\beta}
  = T_{-\alpha}^{-\beta}.
 \end{align*}
 Using this, it easy to check the claim.
\end{proof}

Finally, we provide formulas for commutators.

\begin{lemma}\label{L:commutators}
 For all integers $0 \leqs a,b,m \leqs r-1$, we have
 \begin{align}
  F^{(a)} E^b T_\alpha^\beta &= \sum_{k=0}^{\min \{ a,b \}} \sqbinom{b}{k} \{ a-b-\alpha;k \} E^{b-k} F^{(a-k)} T_\alpha^\beta, \label{E:commutator_Habiro_T_right} \\*
  T_\alpha^\beta F^{(a)} E^b &= \sum_{k=0}^{\min \{ a,b \}} \sqbinom{b}{k} \{ -a+b-\alpha;k \} T_\alpha^\beta E^{b-k} F^{(a-k)}, \label{E:commutator_Habiro_T_left}
 \end{align}
 where $\{ n;k \} := \prod_{j=0}^{k-1} \{ n-j \}$ for all integers $0 \leqs k \leqs n$.
\end{lemma}

\begin{proof}
 If we apply the algebra isomorphism $\omega : \fraku_0^0 \to \fraku_0^0$ defined by
 \begin{align*}
  \omega(E) &= F, &
  \omega(F) &= E, &
  \omega(K) &= K^{-1},
 \end{align*}
 to \cite[Equation~(5), Section~3.1.1]{KS97}, we obtain 
 \begin{align*}
  F^a E^b 
  &= \sum_{k=0}^{\min \{ a,b \}} \sqbinom{a}{k} \sqbinom{b}{k} [k]! E^{b-k} F^{a-k} \left( \prod_{j=0}^{k-1} \frac{q^{a-b-j} K^{-1} - q^{-a+b+j} K}{\{ 1 \}} \right).
 \end{align*}
 which means that
 \begin{align*}
  F^{(a)} E^b T_\alpha^\beta
  &= \sum_{k=0}^{\min \{ a,b \}} \sqbinom{b}{k} E^{b-k} F^{(a-k)} \left( \prod_{j=0}^{k-1} q^{a-b-j} K^{-1} - q^{-a+b+j} K \right) T_\alpha^\beta \\*
  &= \sum_{k=0}^{\min \{ a,b \}} \sqbinom{b}{k} \{ a-b-\alpha;k \} E^{b-k} F^{(a-k)} T_\alpha^\beta.
 \end{align*} 
 This implies
 \begin{align*}
  T_\alpha^\beta F^{(a)} E^b 
  &= F^{(a)} E^b T_{\alpha+2a-2b}^\beta 
  = \sum_{k=0}^{\min \{ a,b \}} \sqbinom{b}{k} \{ -a+b-\alpha;k \} E^{b-k} F^{(a-k)} T_{\alpha+2a-2b}^\beta \\*
  &= \sum_{k=0}^{\min \{ a,b \}} \sqbinom{b}{k} \{ -a+b-\alpha;k \} T_\alpha^\beta E^{b-k} F^{(a-k)}. \qedhere
 \end{align*}
\end{proof}

\subsection{Ribbon structure and factorizability}\label{S:ribbon_factorizability_sl2}

Next, $\fraku$ supports the structure of a ribbon Hopf $G$-bialgebra. Indeed, an \textit{R-matrix} $R_{\bar{\alpha},\bar{\beta}}^{\frac{\bar{\beta}}{2},\frac{\bar{\alpha}}{2}} \in \fraku_{\bar{\alpha}}^{\frac{\bar{\beta}}{2}} \otimes \fraku_{\bar{\beta}}^{\frac{\bar{\alpha}}{2}}$ is given by
\begin{align}
 R_{\bar{\alpha},\bar{\beta}}^{\frac{\bar{\beta}}{2},\frac{\bar{\alpha}}{2}} 
 &= \sum_{a,b,c=0}^{r'-1} \frac{\{ 1 \}^c}{[c]!} q^{\frac{c(c-1)}{2} - 2 \left( b+\frac{\beta}{2} \right) \left( a+\frac{\alpha}{2} \right)} K^{b+\frac{\beta}{2}} E^c \otimes K^{a+\frac{\alpha}{2}} F^c \label{E:R}
\end{align}
for any pair of representatives $\alpha,\beta \in \C$ of $\bar{\alpha},\bar{\beta} \in G$, as shown in \cite[Theorem~4.2]{GHP22}. The R-matrix can also be rewritten as follows:
\begin{align}
 R_{\bar{\alpha},\bar{\beta}}^{\frac{\bar{\beta}}{2},\frac{\bar{\alpha}}{2}} 
 &= \sum_{b,c=0}^{r'-1} q^{\frac{c(c-1)}{2} - \left( b+\frac{\beta}{2} \right) \alpha} K^{b+\frac{\beta}{2}} E^c \otimes \left( \sum_{a=0}^{r'-1} q^{-(2b+\beta)a} K^{a+\frac{\alpha}{2}} \right) F^{(c)} \nonumber \\
 &= \sum_{b,c=0}^{r'-1} q^{\frac{c(c-1)}{2} - \left( b+\frac{\beta}{2} \right) \alpha} K^{b+\frac{\beta}{2}} E^c \otimes T_{2b+\beta}^{\frac{\alpha}{2}} F^{(c)} \label{E:R_T_right} \\
 &= \sum_{a,c=0}^{r'-1} q^{\frac{c(c-1)}{2} - \left( a+\frac{\alpha}{2} \right) \beta} \left( \sum_{b=0}^{r'-1} q^{-(2a+\alpha)b} K^{b+\frac{\beta}{2}} \right) E^c \otimes K^{a+\frac{\alpha}{2}} F^{(c)} \nonumber \\
 &= \sum_{a,c=0}^{r'-1} q^{\frac{c(c-1)}{2} - \left( a+\frac{\alpha}{2} \right) \beta} T_{2a+\alpha}^{\frac{\beta}{2}} E^c \otimes K^{a+\frac{\alpha}{2}} F^{(c)}. \label{E:R_T_left}
\end{align}

A \textit{pivotal element} $g_{\bar{\alpha}} \in \fraku_{\bar{\alpha}}^0$ is given by 
\[
 g_{\bar{\alpha}} := K^{1-r'},
\]
with compatible \textit{ribbon element} $v_{\bar{\alpha}}^{-\bar{\alpha}} := u_{\bar{\alpha}}^{-\bar{\alpha}} (g_{\bar{\alpha}})^{-1} \in \fraku_{\bar{\alpha}}^{-\bar{\alpha}}$.

Furthermore, a \textit{left integral} $\lambda_\alpha \in (\fraku_{\bar{\alpha}}^0)^*$ is given by
\begin{align*}
 \lambda_\alpha \left( E^\ell F^m K^n \right) &= 
 \delta_{\ell,r'-1} \delta_{m,r'-1} \chi_{\frac{r}{2}\Z} \left( n+1 \right) \frac{\sqrt{r'}[r'-1]!}{\{ 1 \}^{r'-1}} q^{\alpha \left( n+1-r' \right)}, \\
 \lambda_\alpha \left( E^\ell F^{(m)} T_{2n+\alpha}^0 \right) 
 &= \frac{\{ 1 \}^m}{r' [m]!} \sum_{b=0}^{r'-1} q^{-(2n+\alpha)b} \lambda_\alpha \left( E^\ell F^m K^b \right) \\
 &= \delta_{\ell,r'-1} \delta_{m,r'-1} \frac{q^{(2n+\alpha)\left( 1-r' \right)}}{\sqrt{r'}}
\end{align*}
for every $\alpha \in \C$, where we denote by $\chi_A$ the indicator function of $A \subset \C$, defined as
\[
 \chi_A(z) = 
 \begin{cases}
  1 & \mbox{ if } z \in A, \\
  0 & \mbox{ if } z \not\in A.
 \end{cases}
\]
This is a direct consequence of \cite[Theorem~4.2]{GHP22} and Remark~\ref{R:factorizable}.

\begin{remark}
 Similarly, we have
 \begin{align*}
  \lambda_\alpha \left( F^{(\ell)} E^m T_{2n+\alpha}^0 \right) 
  &= \sum_{k=0}^{\min \{ \ell,m \}} \sqbinom{m}{k} \{ \ell-m-2n-\alpha;k \} \lambda_\alpha \left( E^{m-k} F^{(\ell-k)} T_{2n+\alpha}^0 \right) \\
  &= \delta_{\ell,r'-1} \delta_{m,r'-1} \frac{q^{(2n+\alpha)\left( 1-r' \right)}}{\sqrt{r'}}, \\
  \lambda_\alpha \left( F^\ell E^m K^n \right) 
  &= \sum_{a=0}^{r'-1} \frac{q^{(2a+\alpha)n}[\ell]!}{\{ 1 \}^\ell} \lambda_\alpha \left( F^{(\ell)} E^m T_{2a+\alpha}^0 \right) \\
  &= \delta_{\ell,r'-1} \delta_{m,r'-1} \frac{[r'-1]!}{\sqrt{r'} \{ 1 \}^{r'-1}} \sum_{a=0}^{r'-1} q^{(2a+\alpha)n+(2a+\alpha)\left( 1-r' \right)}\  \\
  &= \delta_{\ell,r'-1} \delta_{m,r'-1} \chi_{\frac{r}{2}\Z} \left( n+1 \right) \frac{\sqrt{r'}[r'-1]!}{\{ 1 \}^{r'-1}} q^{\left( n+1-r' \right) \alpha}.
\end{align*}
\end{remark}

\begin{lemma}
 A \textit{two-sided cointegral} $\Lambda^\alpha \in \fraku_0^{\bar{\alpha}}$ is given by
 \begin{align*}
  \Lambda^\alpha := 
  \frac{\{ 1 \}^{r'-1}}{\sqrt{r'} [r'-1]!} \sum_{a=0}^{r'-1} E^{r'-1} F^{r'-1} K^{a+\alpha} = \sqrt{r'} E^{r'-1} F^{(r'-1)} T_0^\alpha
 \end{align*}
 for every $\alpha \in \C$.
\end{lemma}

\begin{proof}
 Thanks to Lemma~\ref{L:coproducts_antipodes}, we have 
 \[
  S \left( \Lambda^{\alpha} \right) 
  = \sqrt{r'} S \left( F^{(r'-1)} E^{r'-1} T_0^\alpha \right) 
  = \sqrt{r'} T_0^{-\alpha} E^{r'-1} F^{(r'-1)} 
  = \Lambda^{-\alpha}. 
 \]
 Since
 \begin{align*}
  E \Lambda^{\alpha}
  &= 0 
  = \epsilon(E) \Lambda^{\alpha}, \\
  \Lambda^{\alpha} F
  &= 0
  = \epsilon(F) \Lambda^{\alpha}, \\
  K \Lambda^{\alpha}
  &= \Lambda^{\alpha}
  = \epsilon(K) \Lambda^{\alpha},
 \end{align*}
 the claim follows.
\end{proof}

\begin{proposition}\label{P:sl2_is_factorizable}
 If $r'$ is odd, then 
 \[
  D_{\bar{\alpha},0} \left( \lambda_{-\alpha}^{\phantom{0}} \circ S_{\bar{\alpha}}^0 \right) = \Lambda^\alpha.
 \]
\end{proposition}

\begin{proof}
 \begin{align*} 
  &D_{\bar{\alpha},0} \left( \lambda_{-\alpha}^{\phantom{0}} \circ S_{\bar{\alpha}}^0 \right) \\
  &\hspace*{\parindent} = \sum_{a,b,c,d=0}^{r'-1} q^{\frac{c(c-1)}{2} + \frac{d(d-1)}{2} - (a+b)\alpha} \lambda_{-\alpha}^{\phantom{0}} \left( S_{\bar{\alpha}}^0 \left( K^a F^{(c)} K^b E^d \right) \right) T_{2a}^{\frac{\alpha}{2}} E^c T_{2b}^{\frac{\alpha}{2}} F^{(d)} \\
  &\hspace*{\parindent} = \sum_{a,b,c,d=0}^{r'-1} q^{\frac{c(c-1)}{2} + \frac{d(d-1)}{2} - (a+b)\alpha} \\*
  &\hspace*{2\parindent} \frac{\{ 1 \}^c}{[c]!} q^{ - 2ac + c\alpha} \lambda_{-\alpha}^{\phantom{0}} \left( S_{\bar{\alpha}}^0 \left( F^c K^{a+b} E^d \right) \right) E^c T_{2(a-c)}^{\frac{\alpha}{2}} T_{2b}^{\frac{\alpha}{2}} F^{(d)} \\
  &\hspace*{\parindent} = \sum_{a,b,c,d=0}^{r'-1} q^{\frac{c(c-1)}{2} + \frac{d(d-1)}{2} - (a+b-c)\alpha - 2ac} \frac{\{ 1 \}^c}{[c]!} \\*
  &\hspace*{2\parindent} (-1)^{c+d} q^{c(c-1) - d(d-1)} \lambda_{-\alpha}^{\phantom{0}} \left( E^d K^{-a-b+c-d} F^c \right) \delta_{2(a-c),2b} E^c T_{2b}^\alpha F^{(d)} \\
  &\hspace*{\parindent} = \sum_{b,c,d=0}^{r'-1} (-1)^{c+d} q^{\frac{3c(c-1)}{2} - \frac{d(d-1)}{2} - 2b\alpha - 2(b+c)c} \frac{\{ 1 \}^c}{[c]!} \\*
  &\hspace*{2\parindent} q^{2(2b+d)c - 2d\alpha} \lambda_{-\alpha}^{\phantom{0}} \left( E^d F^c K^{-2b-d} \right) E^c F^{(d)} T_{2(b+d)}^\alpha \\
  &\hspace*{\parindent} = \sqrt{r'} (-1)^{2(r'-1)} q^{r'(r'-1)} E^{r'-1} F^{(r'-1)} T_0^\alpha \\
  &\hspace*{\parindent} = \sqrt{r'} E^{r'-1} F^{(r'-1)} T_0^\alpha. \qedhere
 \end{align*}
\end{proof}

\subsection{Relation with homological representations}\label{S:J_H_is_homological}

Let us suppose from now on that $r \geqs 3$ is odd, and let $\Sigma_g$ be the surface of genus $g$ with one boundary component considered in Section~\ref{S:top_tangles}. In \cite[Theorem~6.1]{DM22}, the first two authors constructed an isomorphism
\begin{equation*}
 \Phi_g^V : \calH_g^{V(r)} \to (\transm{\fraku}_0^0)^{\otimes g},
\end{equation*}
where $\calH_g^{V(r)}$ is a vector space built from the twisted homology of configuration spaces of $\Sigma_g$. This isomorphism was shown to intertwine commuting actions of the mapping class group $\Mod(\Sigma_g)$ and of the small quantum group $\fraku_0^0$. On the one hand, the action of $\Mod(\Sigma_g)$ on $\calH_g^{V(r)}$ is essentially the natural one obtained by lifting homeomorphisms to the corresponding regular covers, up to a correction on coefficients which is analogue, for odd values of $r$, to the one given in \cite[Theorem~D]{BPS21}. On the other hand, the action on $(\transm{\fraku}_0^0)^{\otimes g}$ is the one determined by Lyubashenko's representations \cite[Section~4]{L94}, which are part of the Kerler--Lyubashenko TQFT $J_{H^0_0}$ of Corollary~\ref{C:relation_with_KL} for $H = \fraku$. It turns out that $\calH_g^{V(r)}$ can be deformed using a cohomology class $\cohom \in H^1(\Sigma_g;\C)$, as explained in \cite[Section~6.3.1]{DM22}. The resulting vector space $\calH_g^{\cohom(r)}$ carries an action of the \textit{Torelli subgroup} $\calI(\varSigma_g) < \Mod(\Sigma_g)$ that was conjectured to be isomorphic to the representation that would arise from a construction generalizing Kerler--Lyubashenko TQFTs to cobordisms equipped with cohomology classes, see \cite[Conjecture~6.6]{DM22}. One of the goals of the present work was to build these generalized Kerler--Lyubashenko TQFTs, as a first step towards extending the isomorphism $\Phi_g^V$ to the case of representations depending on cohomology classes. We rephrase \cite[Conjecture~6.6]{DM22} more precisely in this section.

In \cite[Section~6.3.1]{DM22}, the authors define twisted homology groups\footnote{$H^\BM_\ast$ stands for \textit{Borel--Moore homology}.}
\[
 \homol_{n,g}^\formv := H^\BM_n(X_{n,g},Y_{n,g};\phi_{n,g}^\formv),
\]
where
\begin{itemize}
 \item $X_{n,g}$ is the configuration space of $n$ unordered points in $\Sigma_g$ (see \cite[Equation~(1)]{DM22}),
 \item $Y_{n,g}$ is a specific subset of the boundary $\partial X_{n,g}$ (see \cite[Equation~(4)]{DM22}),
 \item $\varphi_{n,g}^\formv : \Z[\pi_{n,g}]  \to \End_{\Z[q,\formv^{\pm 1}]}(W_g^\formv)$ is an action (see \cite[Equation~(31)]{DM22}) of the fundamental group $\pi_{n,g}$ of $X_{n,g}$ onto a module $W_g^\formv$ over a polynomial ring $\Z[q,\formv^{\pm 1}]$.
\end{itemize}
This \textit{twisted} version of the homology of the pair $(X_{n,g},Y_{n,g})$ naturally admits the structure of a $\Z[q,\formv^{\pm 1}]$-module (see \cite[Appendix~A.1]{DM22}).

The polynomial ring appearing above is
\[
\Z[q,\formv^{\pm 1}] := \Z[q,s_1^{\pm 1},t_1^{\pm 1},\ldots,s_g^{\pm 1},t_g^{\pm 1}],
\]
where the formal variables $s_j,t_j$ can be intepreted as the evaluations of a cohomology class against the homology classes $a_j,b_j$ of Section~\ref{S:top_tangles} for every integer $1 \leqs j \leqs g$. Indeed, every cohomology class $\cohom \in H^1(\varSigma_g;\bbC) \cong \Hom_\Z(H_1(\varSigma_g);\bbC)$ determines a ring homomorphism
\begin{align*}
 \Z[q,\formv^{\pm 1}] &\to \bbC \\*
 (s_j,t_j) &\mapsto (q^{\cohom(a_j)},q^{\cohom(b_j)}),
\end{align*}
where $q^z = e^{\frac{2 z \pi \fraki}{r}}$ for every $z \in \C$. If we consider the submodule $\homol^{\formv(r)}_g \subset \homol_{n,g}^\formv$ spanned by \textit{small embedded cycles}, as defined in \cite[Definition~2.3.1]{DM22}, then the vector space
\[
 \homol^{\cohom(r)}_g := \homol^{\formv(r)}_g \otimes_{\Z[q,\formv^{\pm 1}]} \bbC
\]
is endowed with a projective representation
\[
 \bar{\rho}_g^{\cohom(r)} : \calI(\varSigma_g) \to \PGL_{\bbC} \left( \homol^{\cohom(r)}_g \right)
\]
of the Torelli subgroup $\calI(\varSigma_g) \subset \Mod(\Sigma_g)$, as explained in \cite[Section~6.3.1]{DM22}.

Now, let us fix a cohomology class $\cohom \in H^1(\varSigma_g;\bbC)$ and the corresponding object $\bfSigma_{\bar{\cohom}} := (\Sigma_g,\bar{\cohom},\calL_g)$ of $3\Cob^{\C/2\Z}$, where $\bar{\cohom} \in H^1(\varSigma_g;\bbC/2\Z)$ is the image of $\cohom$, and $\calL_g$ is the Lagrangian subspace introduced in Section~\ref{S:handle_surgery_functors}. Let us consider a diffeomorphism $f \in \calI(\Sigma_g)$, and let $\bfSigma_{\bar{\cohom}} \times [0,1]_f := (\varSigma_g \times [0,1]_f,\pi^*(\bar{\cohom}),0)$ be the \textit{mapping cylinder of $f$}, which is a morphism in $3\Cob^{\C/2\Z}$ (see \cite[Section~4.2.3]{DM22} for a precise definition), where $\pi : \varSigma_g \times [0,1] \to \varSigma_g$ is the projection onto the first factor. Notice that $f^*(\bar{\cohom}) = \bar{\cohom}$, since $f \in \calI(\varSigma_g)$. The functor $J_\fraku$ provides a family of projective representations
\begin{align*}
 \bar{\rho}_g^{\fraku,{\bar{\cohom}}} : \calI(\varSigma_g) & \to \PGL_{\bbC}(J_\fraku(\bfSigma_{\bar{\cohom}})) \\*
 f & \to J_\fraku(\bfSigma_{\bar{\cohom}} \times [0,1]_f)
\end{align*}
parametrized by cohomology classes $\bar{\cohom} \in H^1(\varSigma_g;\bbC/2\Z)$ and by roots of unity $q = e^{\frac{2 \pi \fraki}{r}}$. In \cite[Section~4.2.3]{DM22} we explain why this is actually a projective representation, by looking at compositions of elements of $\calI(\varSigma_g)$, and also why it is not in general a linear representation. Notice that, here, the fact that $f$ is an element of the Torelli subgroup is crucial, otherwise $\bfSigma_{\bar{\cohom}} \times [0,1]_f$ would not be an endomorphism of $\bfSigma_{\bar{\cohom}}$. We expect these representations to agree with the homological construction.

\begin{conjecture}\label{C:homology}
 There is an isomorphism
 \[
  \Phi_g^\cohom : \homol^{\cohom(r)}_g \to \fraku_{(\myuline{\bar{\cohom}(a)},\myuline{\bar{\cohom}(b)})}
 \]
 that intertwines the actions of $\calI(\varSigma_g)$ and $\fraku_0^\oplus$, where 
 \[
  (\myuline{\bar{\cohom}(a)},\myuline{\bar{\cohom}(b)}) = (\bar{\cohom}(a_1),\bar{\cohom}(b_1),\ldots,\bar{\cohom}(a_g),\bar{\cohom}(b_g)) \in (\C/2\Z)^{2g}.
 \]
\end{conjecture}

In Conjecture~\ref{C:homology}, the action of $\fraku_0^\oplus$ on $\homol^{\cohom(r)}_g$ has not been defined yet, but should be a natural extension of the action of $\fraku_0^0$ on the twisted homology group of \cite[Theorem~6.1]{DM22}.

This conjecture is motivated by the fact that, when $\cohom$ is the trivial cohomology class, an isomorphism $\Phi_g^V$ was constructed in \cite[Theorem~6.1]{DM22}. We expect all the computations involved in the proof to admit a generalization to the decorated case on both sides of the isomorphism. We expect this conjecture to imply, thanks to the integrality properties of homological representations, that the quantum representations $\bar{\rho}_g^{\fraku,\bar{\cohom}}$ are integral when computed in the integral bases of Section~\ref{S:half-divided_basis}, in the sense that their coefficients can be restricted to the polynomial ring $\Z[q,q^\cohom]$ (defined in \cite[Conjecture~6.6]{DM22}). This would imply a polynomial dependence in $q$ and a polynomial dependence in coefficients of the cohomology class $\cohom$. To the best of our knowledge, this is not known for TQFT functors decorated with cohomology classes. The representations $\bar{\rho}_g^{\fraku,\bar{\cohom}}$ would then form a family that depends polynomially on the cohomology class $\bar{\cohom}$, and that admit both a TQFT formulation, on one hand, and a homological intepretation, on the other.


\begin{thebibliography}{99}

 \bibitem[At88]{A88}
 M.~Atiyah, 
 \textit{Topological Quantum Field Theories}, 
 \doi{10.1007/BF02698547}{Publ. Math. Inst. Hautes Études Sci. \textbf{68} (1988),	no.~1, 175--186}.
 
 \bibitem[BBDP23]{BBDP23}
 A.~Beliakova, I.~Bobtcheva, M.~De~Renzi, R.~Piergallini,
 \textit{On Algebraization in Low-Dimensional Topology},
 \arxiv{2312.15986}{[math.GT]}.
 
 \bibitem[BCGP14]{BCGP14}
 C.~Blanchet, F.~Costantino, N.~Geer, B.~Patureau-Mirand,  
 \textit{Non-Semisimple TQFTs, Reidemeister Torsion and Kashaev's Invariants}, 
 \doi{10.1016/j.aim.2016.06.003}{Adv. Math. \textbf{301} (2016), 1--78};
 \arxiv{1404.7289}{[math.GT]}.
 
 \bibitem[BD21]{BD21}
 A.~Beliakova, M.~De~Renzi, 
 \textit{Kerler--Lyubashenko Functors on $4$-Di\-men\-sion\-al $2$-Han\-dle\-bod\-ies},
 \doi{10.1093/imrn/rnac039}{Int. Math. Res. Not. IMRN (2022), rnac039};
 \arxiv{2105.02789}{[math.GT]}.

 \bibitem[BD22]{BD22}
 A.~Beliakova, M.~De Renzi, 
 \textit{Refined Bobtcheva--Messia Invariants of $4$-Di\-men\-sion\-al $2$-Han\-dle\-bod\-ies},
 \doi{10.4171/irma/34/20}{Essays in Geometry, 387--432, IRMA Lect. Math. Theor. Phys. \textbf{34}, Eur. Math. Soc., Zürich, 2023}; 
 \arxiv{2205.11385}{[math.GT]}.
 
\bibitem[BHMV95]{BHMV95}
 C.~Blanchet, N.~Habegger, G.~Masbaum, P.~Vogel,
 \textit{Topological Quantum Field Theories Derived from the Kauffman Bracket},
 \doi{10.1016/0040-9383(94)00051-4}{Topology \textbf{34} (1995), no.~4, 883--927}.

 \bibitem[Bi00]{B00}
 S.~Bigelow,
 \textit{Braid Groups Are Linear},
 \doi{10.1090/S0894-0347-00-00361-1}{J. Amer. Math. Soc. \textbf{14} (2001), no.~2, 471--486};
 \arxiv{math/0005038}{[math.GR]}.
 	
 \bibitem[BPS21]{BPS21}
 C.~Blanchet, M.~Palmer, A.~Shaukat,
 \textit{Heisenberg Homology on Surface Configurations},
 \arxiv{2109.00515}{[math.GT]}.
 
 \bibitem[CGP12]{CGP12}
 F.~Costantino, N.~Geer, B.~Patureau-Mirand,  
 \textit{Quantum Invariants of 3-Manifolds via Link Surgery Presentations and Non-Semi-Simple Categories}, 
 \doi{10.1112/jtopol/jtu006}{J. Topol. \textbf{7} (2014), no.~4, 1005--1053};
 \arxiv{1202.3553}{[math.GT]}.
 
 \bibitem[CGPV23]{CGPV23}
 F.~Costantino, N.~Geer, B.~Patureau-Mirand, A.~Virelizier,
 \textit{Non Compact $(2+1)$-TQFTs From Non-Semisimple Spherical Categories},
 \arxiv{2302.04509}{[math.QA]}.

 \bibitem[CY94]{CY94}
 L.~Crane, D.~Yetter,
 \textit{On Algebraic Structures Implicit in Topological Quantum Field Theories},
 \doi{10.1142/S0218216599000109}{J. Knot Theory Ramifications \textbf{8} (1999), no.~2, 125--163};
 \href{http://arXiv.org/abs/hep-th/9412025}{\texttt{arXiv:hep-th/9412025}}.

 \bibitem[De17]{D17}
 M.~De~Renzi, 
 \textit{Non-Semisimple Extended Topological Quantum Field The\-o\-ries},
 \doi{10.1090/memo/1364}{Mem. Amer. Math. Soc. \textbf{277} (2022), no.~1364}; 
 \arxiv{1703.07573}{[math.GT]}.
 
 \bibitem[De21]{D21}
 M.~De~Renzi,
 \textit{Extended TQFTs From Non-Semi\-simple Modular Categories},
 \doi{10.1512/iumj.2021.70.9364}{Indiana Univ. Math. J. \textbf{70} (2021), no.~5, 1769--1811};
 \arxiv{2103.04724}{[math.GT]}.
 
 \bibitem[DGGPR19]{DGGPR19}
 M.~De~Renzi, A.~Gainutdinov, N.~Geer, B.~Patureau-Mirand, I.~Runkel,
 \textit{$3$-Di\-men\-sion\-al TQFTs from Non-Sem\-i\-sim\-ple Modular Categories},
 \doi{10.1007/s00029-021-00737-z}{Selecta Math. (N.S.) \textbf{28} (2022), 42}; 
 \arxiv{1912.02063}{[math.GT]}.

 \bibitem[DM22]{DM22}
 M.~De Renzi, J.~Martel,
 \textit{Homological Construction of Quantum Representations of Mapping Class Groups},
 \arxiv{2212.10940}{[math.GT]}.
 
 \bibitem[EGNO15]{EGNO15}
 P.~Etingof, S.~Gelaki, D.~Nikshych, V.~Ostrik,
 \textit{Tensor Categories}, 
 \doi{10.1090/surv/205}{Math. Surveys Monogr. \textbf{205}, Amer. Math. Soc., Providence, RI, 2015}.
%
 \bibitem[GHP20]{GHP20}
 N.~Geer, N.~Ha, B.~Patureau-Mirand,
 \textit{Modified Graded Hennings Invariants From Unrolled Quantum Groups and Modified Integral},
 \doi{10.1016/j.jpaa.2021.106815}{J. Pure Appl. Algebra \textbf{226} (2022), no.~3, 106815};
 \arxiv{2006.12050}{[math.QA]}.
 
 \bibitem[GHP22]{GHP22}
 N.~Geer, N.~Ha, B.~Patureau-Mirand,
 \textit{Modified Symmetrized Integral in $G$-Coalgebras},
 to appear in \doi{10.1142/S0218216523500827}{J. Knot Theory Ramifications};
 \arxiv{2209.04691}{[math.QA]}.
 
 \bibitem[Ha05]{H05}
 K.~Habiro,
 \textit{Bottom Tangles and Universal Invariants},
 \doi{10.2140/agt.2006.6.1113}{Algebr. Geom. Topol. \textbf{6} (2006), no.~3, 1113--1214};
 \arxiv{math/0505219}{[math.GT]}.

 \bibitem[He96]{H96} 
 M.~Hennings, 
 \textit{Invariants of Links and $3$-Manifolds Obtained from Hopf Algebras},
 \doi{10.1112/jlms/54.3.594}{J. London Math. Soc. \textbf{54} (1996), 594--624}.
%
 \bibitem[Ka95]{K95}
 C.~Kassel,
 \textit{Quantum Groups},
 \doi{10.1007/978-1-4612-0783-2}{Grad. Texts in Math. \textbf{155}, Springer-Verlag, New York, 1995}.

 \bibitem[Ke01]{K01}
 T.~Kerler,
 \textit{Towards an Algebraic Characterization of $3$-Dimensional Cobordisms},
 \doi{10.1090/conm/318}{Diagrammatic Morphisms and Applications (San Francisco, CA, 2000), 141--173, Contemp. Math. \textbf{318}, Amer. Math. Soc., Providence, RI, 2003};
 \arxiv{math/0106253}{[math.GT]}.
 
 \bibitem[Ke96]{K96}
 T.~Kerler,
 \textit{Genealogy of Nonperturbative Quantum-Invariants of 3-Manifolds: The Surgical Family}, in 
 \href{https://www.crcpress.com/Geometry-and-Physics/Pedersen-Andersen-Dupont-Swann/p/book/9780824797911}{Geometry and Physics (Aarhus, 1955), 503--547, Lecture Notes in Pure and Appl. Math. \textbf{184}, Dekker, New York, 1997};
 \href{http://arXiv.org/abs/q-alg/9601021}{\texttt{arXiv:q-alg/9601021}}.
 
 \bibitem[KL01]{KL01}
 T.~Kerler, V.~Lyubashenko,
 \textit{Non-Semisimple Topological Quantum Field Theories for 3-Manifolds with Corners},
 \doi{10.1007/b82618}{Lecture Notes in Math. \textbf{1765}, Springer-Verlag, Berlin, 2001}.

 \bibitem[KR93]{KR93}
 L.~Kauffman, D.~Radford,
 \textit{A Necessary and Sufficient Condition for a Finite-Dimensional Drinfel'd Double to Be a Ribbon Hopf Algebra},
 \doi{10.1006/jabr.1993.1148}{J. Algebra \textbf{159} (1993), no.~1, 98--114}.
 
 \bibitem[KS97]{KS97}
 A.~Klimyk, K.~Schmüdgen,
 \textit{Quantum Groups and Their Representations},
 \doi{10.1007/978-3-642-60896-4}{Texts Monogr. Phys., Springer-Verlag, Berlin, 1997}.

 \bibitem[Ly94]{L94}
 V.~Lyubashenko,
 \textit{Invariants of $3$-Manifolds and Projective Representations of Mapping Class Groups via Quantum Groups at Roots of Unity},
 \doi{10.1007/BF02101805}{Comm. Math. Phys. \textbf{172} (1995), no.~3, 467--516};
 \href{http://arXiv.org/abs/hep-th/9405167}{\texttt{arXiv:hep-th/9405167}}.
 
 \bibitem[Ma20]{M20}
 J.~Martel,
 \textit{A Homological Model for $U_q \fsl(2)$ Verma-Modules and Their Braid Representations},
 \doi{10.2140/gt.2022.26.1225}{Geom. Topol. \textbf{26} (2022), no.~3, 1225--1289};
 \arxiv{2002.08785}{[math.GT]}.
 
 \bibitem[Ma91]{M91}
 S.~Majid,
 \textit{Braided Groups and Algebraic Quantum Field Theories},
 \doi{10.1007/BF00403542}{Lett. Math. Phys. \textbf{22} (1991), no.~3, 167--175}
 
 \bibitem[Ma95]{M95}
 S.~Majid, 
 \textit{Foundations of Quantum Group Theory},
 \doi{10.1017/CBO9780511613104}{Cambridge University Press, Cambridge, 1995}.

 \bibitem[Oh93]{O93}
 T.~Ohtsuki,
 \textit{Colored Ribbon Hopf Algebras and Universal Invariants of Framed Links},
 \doi{10.1142/S0218216593000131}{J. Knot Theory Ramifications \textbf{2} (1993), no.~2, 211--232}.
 
 \bibitem[Ra12]{R12}
 D.~Radford,
 \textit{Hopf Algebras},
 \doi{10.1142/8055}{Ser. Knots Everything \textbf{49}, World Sci. Publ.,
Hackensack, NJ, 2012.}

 \bibitem[Ro97]{R97}
 J.~Roberts,
 \textit{Kirby Calculus in Manifolds With Boundary},
 \href{https://journals.tubitak.gov.tr/math/abstract.htm?id=1031}{Turkish J. Math. \textbf{21} (1997), no.~1, 111--117};
 \href{https://arXiv.org/abs/math/9812086}{\texttt{arXiv:math/9812086} [math.GT]}.
 
 \bibitem[RT91]{RT91}
 N.~Reshetikhin, V.~Turaev,
 \textit{Invariants of 3-Manifolds via Link Polynomials and Quantum Groups},
 \doi{10.1007/BF01239527}{Invent. Math. \textbf{103} (1991), no.~1, 547--597}.
 
 \bibitem[Tu10]{T10}
  V.~Turaev,
  \textit{Homotopy Quantum Field Theory},
  \doi{10.4171/086}{EMS Tracts Math. \textbf{10}, European Mathematical Society (EMS), Zürich, 2010}.
  
 \bibitem[Tu94]{T94} 
 V.~Turaev,
 \textit{Quantum Invariants of Knots and $3$-Man\-i\-folds},
 \href{https://www.degruyter.com/view/product/461906}{De Gruyter Stud. Math. \textbf{18}, De Gruyter, Berlin, 1994}.
 
 \bibitem[Vi00]{V00}
 A.~Virelizier,
 \textit{Hopf Group-Coalgebras},
 \doi{10.1016/S0022-4049(01)00125-6}{J. Pure Appl. Algebra \textbf{171} (2002), no.~1, 75--122},
 \href{https://arXiv.org/abs/math/0012073}{\texttt{arXiv:math/0012073} [math.QA]}.
 
 \bibitem[Wi89]{W89}
 E.~Witten,
 \textit{Quantum Field Theory and the Jones Polynomial},
 \doi{10.1007/BF01217730}{Comm. Math. Phys. \textbf{121} (1989), no.~3, 351--399}.
 
\end{thebibliography}
\end{document}